\documentclass[11pt]{article}
\usepackage[numbers,square]{natbib}
\usepackage{amsmath}
\usepackage{amsthm}
\usepackage{amsfonts}
\usepackage{amssymb}
\usepackage{siunitx}
\usepackage{commath}
\usepackage{graphicx}
\usepackage{xspace}
\usepackage{color}
\usepackage[lofdepth,lotdepth]{subfig}
\usepackage{stmaryrd}
\usepackage{url}
\usepackage{amsthm}
\usepackage{graphbox}
\usepackage{footnote}
\makesavenoteenv{tabular}
\makesavenoteenv{table}
\usepackage[hidelinks]{hyperref}
\hypersetup{colorlinks = true, urlcolor = blue,
  linkcolor = blue, citecolor = blue}
\usepackage{cleveref}

\usepackage[margin=0.7in]{geometry}
\makeatletter
\newcommand{\tnorm}{\@ifstar\@tnorms\@tnorm}
\newcommand{\@tnorms}[1]{%
  \left|\mkern-1.5mu\left|\mkern-1.5mu\left|
   #1
  \right|\mkern-1.5mu\right|\mkern-1.5mu\right|
}
\newcommand{\@tnorm}[2][]{%
  \mathopen{#1|\mkern-1.5mu#1|\mkern-1.5mu#1|}
  #2
  \mathclose{#1|\mkern-1.5mu#1|\mkern-1.5mu#1|}
}
\makeatother

\newcommand{\jump}[1]{\llbracket #1 \rrbracket}
\newcommand{\av}[1]{\{\!\!\{#1\}\!\!\}}
\usepackage{xspace,color}

\newtheorem{theorem}{Theorem}
\newtheorem{lemma}{Lemma}

\newtheorem{remark}{Remark}
\title{A coupled HDG/DG method for porous media with conducting/sealing faults}
\author{A. Cesmelioglu\thanks{Department of Mathematics and
    Statistics, Oakland University, MI, USA
    (\url{cesmelio@oakland.edu}),
    \url{https://orcid.org/0000-0001-8057-6349}}
  \and
  M. Kuchta\thanks{Simula Research Laboratory, Oslo, Norway
    (\url{miroslav@simula.no}),
    \url{https://orcid.org/0000-0002-3832-0988}}
  \and
  J. J. Lee\thanks{Department of Mathematics, Baylor University,
    TX, USA (\url{jeonghun_lee@baylor.edu}),
    \url{https://orcid.org/0000-0001-5201-8526}}
  \and
  S. Rhebergen\thanks{Department of Applied Mathematics, University of
    Waterloo, ON, Canada (\url{srheberg@uwaterloo.ca}),
    \url{http://orcid.org/0000-0001-6036-0356}}}
\begin{document}
%------------------------------------------------------------------------------
\maketitle
%------------------------------------------------------------------------------
\begin{abstract}
  We introduce and analyze a coupled hybridizable discontinuous
  Galerkin/discontinuous Galerkin (HDG/DG) method for porous media in
  which we allow fully and partly immersed faults, and faults that
  separate the domain into two disjoint subdomains. We prove
  well-posedness and present an a priori error analysis of the
  discretization. Numerical examples verify our analysis.
\end{abstract}
%------------------------------------------------------------------------------
\section{Introduction}
\label{s:introduction}

Subsurface flow problems are of interest to many areas of science and
engineering such as geophysics, environmental sciences, hydrocarbon
extraction, and geothermal energy production. Faults are geological
structures that are discontinuities of displacement. In subsurface
flow problems, faults can act as conduits or barriers to fluid flow,
depending on the permeability on faults. These fault structures can
cause significant changes in fluid flows, so understanding the
interplay of faults (as conduits or as barriers) and fluid flows is
important for applications. For the remainder of this paper we will
refer to conduits (often called fractures) as {\it conducting faults}
and barriers as {\it sealing faults}.

A mathematical model of subsurface flows with conducting and sealing
faults was proposed in \cite{Martin:2005}. They furthermore analyzed a
mixed finite element method for this problem. After this seminal work,
numerous works on discretizing subsurface flows with faults have
appeared in the literature. These include the hybrid high order method
\cite{Chave:2019}, the interior penalty discontinuous Galerkin method
\cite{liu2024interior}, staggered discontinuous Galerkin methods
\cite{Zhao:2022}, a hybridized interior penalty method
\cite{Leng:2024}, a mixed virtual element method
\cite{Benedetto:2022}, a discrete finite volume method \cite{Li:2021},
a cut finite element method \cite{Burman:2020}, a multipoint flux
approximation method \cite{Cavalcante:2020}, a mixed hybrid mortar
method \cite{Pichot:2010,Pichot:2012}, and a finite volume method
\cite{Chen:2019}.

The Darcy equations for the porous-matrix flow are defined on a
$\dim$-dimensional domain. Fluid flow in faults, however, are modelled
as flow problems on $(\dim-1)$-dimensional domains. In this paper, to
discretize this inter-dimensional problem, we propose a coupled dual
mixed hybridizable discontinuous Galerkin (HDG) method and interior
penalty discontinuous Galerkin (IPDG) method. The HDG method was
originally introduced in \cite{Cockburn:2009a} as an approach to
reduce the computational costs of traditional discontinuous Galerkin
methods. This was achieved by introducing new face unknowns in the
discretization in such a way as to facilitate static condensation. The
introduction of these new face unknowns, defined on
$(\dim-1)$-dimensional faces of the mesh, and their coupling to cell
unknowns defined on $\dim$-dimensional cells of the mesh, however,
also presents a natural framework to deal with the inter-dimensional
problem of porous-matrix flow with flow in faults. The
$\dim$-dimensional Darcy equations for the porous-matrix flow are
discretized using the dual mixed HDG method (also called the LDG-H
method in \cite{Cockburn:2009a}). On $(\dim-1)$-dimensional conducting
faults we discretize the flow equations using the IPDG method
\cite{Arnold:2002,Arnold:1982}. The coupling between the
$\dim$-dimensional HDG and $(\dim-1)$-dimensional IPDG method is
handled automatically by the HDG framework. Finally, sealing faults
are easily included in the HDG discretization as modified interface
conditions.

This paper is organized as follows. In \cref{s:fault,s:hdgipdg} we
introduce the governing equations and its coupled HDG/DG finite
element discretization, respectively. In
\cref{s:wellposedness,s:erroranalysis}, we prove well-posedness of the
discretization and present an a priori error estimate for the
numerical solution. Finally, in \cref{s:numericalresults} we present
numerical results which include test cases that illustrate our
theoretical analysis, as well as some benchmark test cases. We
conclude in \cref{s:conclusions}.

%------------------------------------------------------------------------------
\section{Porous media with faults}
\label{s:fault}

Let $\Omega$ be an open, bounded, connected domain in
$\mathbb{R}^{\dim}$ with $\dim \in \cbr[0]{2,3}$. Let $\Gamma_c$,
$\Gamma_s$ be unions of disjoint $(\dim-1)$-dimensional piecewise
linear submanifolds in $\Omega$. We will refer to $\Gamma_c$ as a
conducting fault and $\Gamma_s$ as a sealing fault.  Denote by
$\partial \Omega$ the polygonal/polyhedral boundary of $\Omega$ and by
$\partial\Gamma_j$ the boundary of $\Gamma_j$, $j = c,s$. We will
consider faults that are fully immersed in $\Omega$ (if
$x \in \partial\Gamma_j$ then $x \notin \partial\Omega$), partly
immersed in $\Omega$ (there exist $x,y \in \partial \Gamma_j$ such
that $x \in \partial\Omega$ and $y \notin \partial \Omega$), and
faults that separate the domain into two disjoint subdomains (if
$x \in \partial\Gamma_j$ then $x \in \partial\Omega$). Let
$\mathring{\Omega}:=\Omega\setminus (\Gamma_c \cup \Gamma_s)$ and $n$
denote the outward unit normal vector to $\Omega$.

Let $\gamma$ be a subset of $\Gamma_j$, $j=c,s$, of positive
$(\dim-1)$-dimensional Lebesgue measure which is obtained by the
intersection of a $(\dim-1)$-dimensional plane and $\Gamma_j$, and let
$\Omega_+^{\gamma}$, $\Omega_-^{\gamma}$ be two disjoint subdomains in
$\Omega$ such that
$\gamma = \partial \Omega_+^{\gamma} \cap \partial
\Omega_-^{\gamma}$. On $\gamma$, $n_{\pm}$ is the unit outward normal
vector field on $\partial \Omega_{\pm}^{\gamma}$. Note that $n_{+}$
($n_-$, resp.) is independent of the choice of $\Omega_+^{\gamma}$
($\Omega_-^{\gamma}$, resp.). Let
$v \in [L^2(\Omega_+^{\gamma} \cup \Omega_-^{\gamma})]^{\dim}$ be such
that its traces on $\gamma$ from $\Omega_{\pm}^{\gamma}$, denoted by
$v_{\pm}$, lie in $[L^2(\gamma)]^{\dim}$. Then, we define
$\jump{v\cdot n}_{\gamma}:=(v_{+}-v_{-})\cdot n_{+}$,
$\av{v\cdot n}_{\gamma} := \tfrac{1}{2}(v_{+} + v_{-}) \cdot
n_{+}$. Similarly, let
$q \in L^2(\Omega_+^{\gamma} \cup \Omega_-^{\gamma})$ such that its
traces on $\gamma$ from $\Omega_{\pm}^{\gamma}$, denoted by $q_{\pm}$,
lie in $L^2(\gamma)$. Then $\jump{q}_{\gamma}:=q_{+}-q_{-}$,
$\av{q}_{\gamma} := \tfrac{1}{2}(q_{+} + q_{-})$. Note that
$\jump{v\cdot n}_{\gamma}$, $\av{v\cdot n}_{\gamma}$,
$\jump{q}_{\gamma}$, $\av{q}_{\gamma}$ are independent of the choices
of $\Omega_+^{\gamma}$, $\Omega_-^{\gamma}$. If $\gamma$ is clear from
the context, then we will write $\jump{v\cdot n}$, $\av{v\cdot n}$,
$\jump{q}$, $\av{q}$ for simplicity.  We depict the domain notation in
\cref{fig:Domain}.

\begin{figure}[t]
  \begin{center}
    \includegraphics[height=0.5\textwidth]{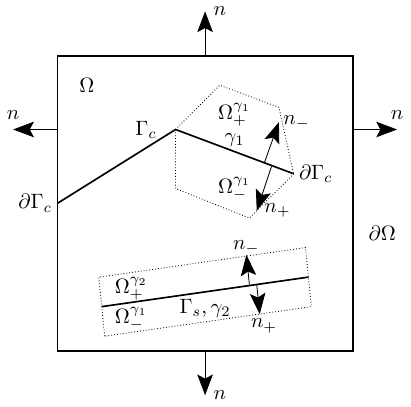}    
    \caption{Illustration of the domain notation in $\mathbb{R}^2$.}
    \label{fig:Domain}
  \end{center}
\end{figure}

In $\Omega$ we denote the Darcy velocity by $u$, the pressure by $p$,
a source term by $g$, and the permeability tensor by
$\kappa$. Following \cite{Martin:2005}, we will assume that $\kappa$
is diagonal and there exist constants $\kappa_{\min}, \kappa_{\max}$
such that $0 < \kappa_{\min} \le \kappa_{jj} \le \kappa_{\max}$ for
$j=1,\hdots,\text{dim}$ almost everywhere in $\Omega$.

The velocity, pressure, permeability, and source/sink term in the
fault $\Gamma_c$ are denoted by $u_f$, $p_f$, $\overline{\kappa}_f$,
and $g_f$, respectively. Next, suppose that a unit normal vector $n$
at $x \in \Gamma_c \cup \Gamma_s$ is uniquely determined up to
orientation and define
$\overline{\kappa}_{f,n}(x):=\kappa(x) n\cdot n$. Let $\tau_i$,
$1 \le i \le \dim$ be an orthonormal basis of the tangent space along
$\Gamma_c \cup \Gamma_s$ at $x$. We can write
$\bar{\kappa}_{f,\tau}(x)$ in terms of this basis:
$\bar{\kappa}_{f,\tau}(x) = \sum_{i=1}^{\dim-1} \bar{\kappa}_i(x)
\tau_i \otimes \tau_i$ with $\bar{\kappa}_{i}(x)$ the expansion
coefficients. Then, for almost all $x \in \Gamma_c \cup \Gamma_s$, we
assume that
$\kappa(x) = \bar{\kappa}_{f,n}(x) n \otimes n + \bar{\kappa}_{f,
  \tau}(x)$. We further assume that
$\kappa_{\min} \le \bar{\kappa}_{f,n}(x), \bar{\kappa}_i(x) \le
\kappa_{\max}$ for almost all $x$. Finally, we define
$\kappa_f := \overline{\kappa}_{f,\tau}d$ and
$\alpha_f := 2\overline{\kappa}_{f,n}/d$, where $d$ is the thickness
of the fault.

The porous media model with conducting fault $\Gamma_c$ (cf.
\cite{Martin:2005}) and sealing fault $\Gamma_s$
(cf. \cite{lee2022forward}), is now given by:
\begin{subequations}
  \label{eq:fault_problem}
  \begin{align}
    \label{eq:fault_problem_a}
    \kappa^{-1} u + \nabla p & =0 && \text{in } \mathring{\Omega},
    \\
    \label{eq:fault_problem_b}
    \nabla \cdot u &= g && \text{in } \mathring{\Omega},
    \\
    \label{eq:fault_problem_c}
    -\nabla_\tau \cdot  (\kappa_f\nabla_\tau p_f) &=g_f+\jump{u\cdot n} && \text{in } \Gamma_c,
    \\
    \label{eq:fault_problem_d}
    -\xi u_+ \cdot n_+ +\alpha_f p_+&=\alpha_f p_f -(1-\xi) u_- \cdot n_- && \text{in } \Gamma_c,
    \\
    \label{eq:fault_problem_e}
    -\xi u_- \cdot n_- +\alpha_f p_-&=\alpha_f p_f -(1-\xi) u_+ \cdot n_+ && \text{in } \Gamma_c,
    \\
    \label{eq:fault_problem_f}
    2\av{u \cdot n} &=\alpha_f \jump{p} && \text{in } \Gamma_s,
    \\
    \label{eq:fault_problem_g}
    \jump{u \cdot n} &= 0 && \text{in } \Gamma_s,
  \end{align}
\end{subequations} 
where $1/2 < \xi \le 1$ is a constant and where $\nabla_{\tau}$ and
$\nabla_{\tau}\cdot$ are the tangential gradient and divergence
operators along $\Gamma_c$.

The boundary of the conducting fault $\Gamma_c$ is partitioned into a
Neumann part and a Dirichlet part. These are denoted by
$(\partial\Gamma_c)^N$ and $(\partial\Gamma_c)^D$, respectively.  Let
$p^D$ be a given pressure on $\partial \Omega$, let $p_f^D$ be a given
pressure on $(\partial\Gamma_c)^D$, and let $g_f^N$ be given Neumann
boundary data on $(\partial\Gamma_c)^N$. We then consider the
following boundary conditions:
\begin{subequations}
  \label{eq:fault_problem_bcs}
  \begin{align}
    \label{eq:fault_problem_h}
    p &= p^D && \text{on }\partial\Omega,
    \\
    \label{eq:fault_problem_i}
    -\kappa_f \nabla_\tau p_f &= g^N_{f} && \text{on } (\partial\Gamma_c)^N,
    \\
    \label{eq:fault_problem_j}
    p_f&=p^D_{f} && \text{on } (\partial\Gamma_c)^D.
  \end{align}
\end{subequations} 

%------------------------------------------------------------------------------
\section{The discretization}
\label{s:hdgipdg}

Let $\mathcal{T}$ be a shape-regular (cf. \cite{brenner:book}
triangulation of the domain $\Omega$ into non-overlapping simplices
$K$ that aligns with the faults.  The set of all
$(\dim-1)$-dimensional faces of $\partial K$, for $K \in \mathcal{T}$,
is denoted by $\mathcal{F}$. We assume that the intersection of
$\Gamma_j$ and $\mathcal{F}$ gives a triangulation of $\Gamma_j$ for
$j=c,s$. For a $k$-dimensional simplex $S$, $h_S$ is the diameter of
the smallest $k$-dimensional ball containing $S$, and we define
$h := \max_{K \in \mathcal{T}} h_K$.  The outward unit normal vector
field on $\partial K$ for a simplex $K$ is denoted by $n_K$. However,
if $K$ is clear from the context, we will simply write $n$.

We split $\mathcal{F}$ into four sets; $\mathcal{F}_{b}$ for the
boundary faces, $\mathcal{F}_{c}^f$ and $\mathcal{F}_{s}^f$ for the
faces on the conducting and sealing faults, respectively, and
$\mathcal{F}_{0}$ for the interior faces that are not on the
faults. In other words,
\begin{align*}
  \mathcal{F}_{j}^f &:=\cbr[0]{ F \in \mathcal{F}\,:\, F \subset \Gamma_j}, \; j=s,c,
  &
    \mathcal{F}_{b} &:=\cbr[0]{ F \in \mathcal{F}\,:\, F \subset \partial \Omega },
  \\
  \mathcal{F}_0 &:= \mathcal{F} \setminus \del[0]{\mathcal{F}_c^f \cup \mathcal{F}_s^f \cup \mathcal{F}_{b}}.
  &&
\end{align*}
We will also denote the union of all faces in $\mathcal{F}_0$ by
$\Gamma_0 := \bigcup_{F\in \mathcal{F}_0}F$.  Let
$F^+,F^- \in \mathcal{F}_c^f$ be two adjacent faces on the conducting
fault $\Gamma_c$. These faces share $e$, an interior edge when
$\dim=3$ and a vertex when $\dim=2$. An edge/vertex of a face
$F \in \mathcal{F}_c^f$ that lies on $\partial \Gamma_c$ is considered
as a boundary edge/vertex.  The set of all interior edges/vertices on
the conducting fault $\Gamma_c$ is denoted by $\mathcal{E}_c^{f,int}$
while the sets of edges/vertices on $(\partial\Gamma_c)^N$ and
$(\partial\Gamma_c)^D$ are denoted by $\mathcal{E}_c^{f,N}$ and
$\mathcal{E}_c^{f,D}$, respectively. Furthermore, we define
$\mathcal{E}_c^{f,int,D} := \mathcal{E}_c^{f,int} \cup
\mathcal{E}_c^{f,D}$

The outward unit normal vector to a face $F\in \mathcal{F}_c^f$ on one
of its vertices ($\dim=2$) or edges ($\dim=3$) $e$ is denoted by
$n_e$, which reduces to $\pm 1$ if $\dim=2$. For
$e \in \mathcal{E}_c^{f,int}$ shared by faces
$F_+,F_- \in \mathcal{F}_c^f$, we define
\begin{equation}
  \label{eq:htilde-def}
  \tilde{h}_e :=
  \begin{cases}
    h_e & \text{if } \dim=3,
    \\
    \max(h_{F_+},h_{F_-}) & \text{if } \dim = 2.
  \end{cases}
\end{equation}
We define the jump and average operators across an interior edge $e$
shared by faces $F_+,F_- \in \mathcal{F}_c^f$ as
$\jump{pn}_e := p_+n_{e,+} + p_-n_{e,-}$ and
$\av{p}_e := (p_+ + p_-)/2$, respectively. Here $p_{\pm}$ is the trace
of $p|_{F_{\pm}}$ restricted to $e$. On a boundary edge
$e \in \partial \Gamma_c$ we define $\jump{pn}_e := pn_e$ and
$\av{p}_e := p$ with $n_e$ the outward unit normal vector on $e$ to
$\Gamma_c$.

For the discretization of the porous medium with faults problem
\cref{eq:fault_problem,eq:fault_problem_bcs} we consider the following
finite element spaces:
\begin{equation}
  \label{eq:fem_spaces}
  \begin{aligned}
    V_h &:= \cbr[0]{v_h \in [L^2(\Omega)]^{\dim} \,:\, v_h \in [P_k(K)]^{\dim} \ \forall K \in \mathcal{T}},
    \\
    Q_h &:= \cbr[0]{q_h \in L^2(\Omega) \,:\, q_h \in P_{k}(K) \ \forall K \in \mathcal{T}},
    \\
    Q_h^f &:= \cbr[0]{q_h^f \in L^2(\Gamma_c) \,:\, q_h^f \in P_{k_f}(F) \ \forall F \in \mathcal{F}_c^f}, \quad k_f=k,k+1,
    \\
    \bar{Q}_h &:= \cbr[0]{\bar{q}_h \in L^2((\cup_{F \in \mathcal{F}} F)\backslash \Gamma_c) \,:\, \bar{q}_h \in P_k(F) \ \forall F \in \mathcal{F}\backslash \mathcal{F}_c^f},
    \\
    \bar{Q}_h(w) &:= \cbr[0]{\bar{q}_h \in \bar{Q}_h\, :\, \bar{q}_h = \bar{\Pi}_hw \text{ on } \partial\Omega}, \forall w\in L^2(\partial \Omega),
    \\
    \boldsymbol{Q}_h(w) &:= Q_h \times \bar{Q}_h(w) \times Q_h^f,
  \end{aligned}
\end{equation}
where $P_l(R)$ denotes the space of polynomials of degree $l \ge 1$ on
a domain $R$ and $\bar{\Pi}_h$ is the $L^2$-projection into
$\cbr[0]{\bar{q}_h\in L^2(\partial \Omega): \bar{q}_h \in P_k(F)
  \quad\forall F\in \mathcal{F}_{\partial}}$ where
$\mathcal{F}_{\partial}$ is the set of all faces that lie on
$\partial \Omega$. Elements in $\boldsymbol{Q}_h(w) $ will be denoted
by $\boldsymbol{q}_h = (q,\bar{q}, q^f)$.

To write the coupled HDG/DG discretization we introduce bilinear forms
\begin{align*}
  &a_h(\cdot ,\cdot ) : V_h \times V_h \to \mathbb{R},
  & & b_h(\cdot, \cdot) : V_h \times \boldsymbol{Q}_h(w) \to \mathbb{R},
  \\
  &c_h^f(\cdot, \cdot): Q_h^f \times Q_h^f \to \mathbb{R},
  & & c_h(\cdot, \cdot): \boldsymbol{Q}_h(w) \times \boldsymbol{Q}_h(w) \to \mathbb{R}
\end{align*}
by
\begin{subequations}
  \begin{align}
    \label{eq:ah}
    a_h(w, v)
    :=&
        \sum_{K\in \mathcal{T}} \int_K \kappa^{-1} w \cdot v \dif x
        + \sum_{F \in \mathcal{F}_s^f} \int_{F} 2\alpha_f^{-1} \av{w\cdot n} \av{v \cdot n} \dif s
    \\
    \notag 
      & + \sum_{F \in \mathcal{F}_c^f} \int_{F} \alpha_f^{-1}\del[2]{(\xi-\tfrac{1}{2})\jump{w \cdot n}\jump{v \cdot n}
        + 2 \av{w \cdot n} \av{v \cdot n} } \dif s,
    \\
    \label{eq:bh}
    b_h(v, \boldsymbol{q})
    :=&
        - \sum_{K\in \mathcal{T}} \int_K q \nabla \cdot v  \dif x
        + \sum_{K\in \mathcal{T}} \int_{\partial K \backslash \Gamma_c} \bar{q} v \cdot n \dif s
        + \sum_{F \in \mathcal{F}_c^f} \int_{F} q_f \jump{v \cdot n} \dif s,
    \\
    \label{eq:ch-f}
    c_h^f(r, q)
    :=& \sum_{F \in \mathcal{F}_c^f} \int_F \kappa_f \nabla_{\tau} r \cdot \nabla_{\tau} q \dif s
        + \sum_{e \in \mathcal{E}_c^{f,int,D}} \int_e \frac{\sigma}{\tilde{h}_e} \av{\kappa_f}_e\jump{r}_e\jump{q}_e \dif l
    \\
    \notag 
      & - \sum_{e \in \mathcal{E}_c^{f,int,D}} \int_e \av{\kappa_f\nabla_{\tau}r}_e \cdot \jump{q n_e}_e \dif l
        - \sum_{e \in \mathcal{E}_c^{f,int,D}} \int_e \av{\kappa_f\nabla_{\tau}q}_e \cdot \jump{r n_e}_e \dif l,
    \\
    \label{eq:ch}
    c_h(\boldsymbol{r}, \boldsymbol{q})
    :=& c_h^f(r_f, q_f) + \sum_{K\in\mathcal{T}} \int_{\partial K \backslash (\Gamma_c \cup \Gamma_s)} \alpha (q - \bar{q}) (r-\bar{r} ) \dif s,
  \end{align}
\end{subequations}
where for any $K\in \mathcal{T}$, $\alpha\geq 0$ is piecewise constant
on each $F\subset\partial K$ such that $\alpha\neq 0$ on at least one
$F\subset\partial K$ and a linear functional
$f_h: Q_h \times Q_h^f \to \mathbb{R}$ is defined as 
\begin{equation}
  \label{eq:fh}
  \begin{split}
    f_h((q, q_{f}))
    :=& \sum_{K\in \mathcal{T}}\int_K g q \dif x + \sum_{F \in \mathcal{F}_c^f} \int_F g_f q_{f} \dif s
    \\
    & - \sum_{e \in \mathcal{E}_c^{f,N}} \int_e q_{f}g_f^N \cdot n_e \dif l
    + \sum_{e \in \mathcal{E}_c^{f,D}} \int_e (\frac{\sigma}{\tilde{h}_e} \kappa_f p_f^D q_{f}
    - \kappa_f\nabla_{\tau}q_{f}\cdot n_e p_f^D)\dif l.          
  \end{split}
\end{equation}
The coupled HDG/DG discretization for the fault problem
\cref{eq:fault_problem,eq:fault_problem_bcs} is now given by: Find
$(\boldsymbol{u}_h,\boldsymbol{p}_h) \in \boldsymbol{V}_h \times
\boldsymbol{Q}_h(p^D)$ such that
\begin{subequations}
  \label{eq:discrete_fault}
  \begin{align}
    \label{eq:discrete_fault_a}
    a_h({u}_h, {v}_h) + b_h({v}_h, \boldsymbol{p}_h)
    &= 0 && \forall {v}_h \in {V}_h,
    \\
    \label{eq:discrete_fault_b}
    -b_h({u}_h, \boldsymbol{q}_h) + c_h(\boldsymbol{p}_h, \boldsymbol{q}_h)
    &= f_h((q_{h}, q_{h,f})) && \forall \boldsymbol{q}_h \in \boldsymbol{Q}_h(0).
  \end{align}  
\end{subequations}
The following lemma shows that the discretization
\cref{eq:discrete_fault} is consistent.

\begin{lemma}[Consistency]
  \label{lem:consistency}
  Let $u \in [H^1(\mathring{\Omega})]^{\dim}$,
  $p\in H^1(\mathring{\Omega})$, and $p_f \in H^{s}(\Gamma_c)$ with
  $s > 3/2$ satisfy
  \cref{eq:fault_problem,eq:fault_problem_bcs}. Denote the average of
  $p$ on $\mathcal{F}_0 \cup \mathcal{F}_s^f$ by $\av{p}$.
  Then $u$
  and $\boldsymbol{p}:=(p,\av{p}, p_f)$ satisfy
  \cref{eq:discrete_fault}.
\end{lemma}
\begin{proof}
  We will show that $R=0$, where
  \begin{equation*}
    R = a_h({u}, {v}_h) + b_h({v}_h, \boldsymbol{p})
    -b_h({u}, \boldsymbol{q}_h) + c_h(\boldsymbol{p}, \boldsymbol{q}_h)
    -f_h((q_{h}, q_{h,f})).
  \end{equation*}
  We first show that $a_h({u}, {v}_h) + b_h({v}_h, \boldsymbol{p})=0$.
  Multiply \cref{eq:fault_problem_a} by a test function
  $v_{h} \in V_h$, integrate over a cell $K \in \mathcal{T}$,
  integrate by parts, and sum over all simplices in $\mathcal{T}$:
  \begin{multline*}
    \sum_{K \in \mathcal{T}} \int_K \kappa^{-1} u \cdot v_{h} \dif x
    - \sum_{K \in \mathcal{T}} \int_K p \nabla \cdot v_{h} \dif x
    + \sum_{K \in \mathcal{T}} \int_{\partial K \backslash (\Gamma_c\cup \Gamma_s)} p v_{h} \cdot n \dif s
    \\
    +  \sum_{K\in \mathcal{T}} \int_{\partial K \cap (\Gamma_s\cup\Gamma_c)} p v_{h} \cdot n\dif s = 0.
  \end{multline*}
  On $\Gamma_s$ and $\Gamma_c$, using the identity
  $p_+ v_{h,+} \cdot n_+ + p_- v_{h,-} \cdot n_- =
  \av{p}\jump{v_h\cdot n} + \jump{p}\av{v_h\cdot n}$, we have
  \begin{multline}
    \label{eq:ab-form-consistency}
    \sum_{K \in \mathcal{T}} \int_K \kappa^{-1} u \cdot v_{h} \dif x
    - \sum_{K \in \mathcal{T}} \int_K p \nabla \cdot v_{h} \dif x
    + \sum_{K \in \mathcal{T}} \int_{\partial K \backslash (\Gamma_c\cup \Gamma_s)} p v_{h} \cdot n \dif s
    \\
    +  \int_{\Gamma_c\cup \Gamma_s} (\av{p}\jump{v_h\cdot n} + \jump{p}\av{v_h\cdot n}) \dif s = 0.
  \end{multline}
  By \cref{eq:fault_problem_f},
  \begin{align}
    \label{eq:Gamma_s_identity}
    \int_{\Gamma_s}  (\av{p}\jump{v_h\cdot n} + \jump{p}\av{v_h\cdot n}) \dif s
    = \int_{\Gamma_s} (\av{p}\jump{v_h\cdot n}+2\alpha_f^{-1} \av{u\cdot n} \av{v_h\cdot n} ) \dif s. 
  \end{align}
  Note that the sum and difference of \cref{eq:fault_problem_d} and
  \cref{eq:fault_problem_e} give, respectively,
  \begin{align*}
    (1-2\xi)\jump{u\cdot n} + 2\alpha_f \av{p} = 2\alpha_f p_f,
    \quad \alpha_f \jump{p} = 2\av{u\cdot n} \qquad \text{ on } \Gamma_c.
  \end{align*}
  From these identities,
  \begin{equation}
    \label{eq:Gamma_c_identity}
    \begin{split}
      \int_{\Gamma_c}  (\av{p}\jump{v_h\cdot n} + \jump{p}\av{v_h\cdot n}) \dif s  
      =& \sum_{F \in \mathcal{F}_c^f} \int_{F} \alpha_f^{-1} \del[2]{(\xi - \tfrac{1}{2})\jump{u\cdot n}\jump{v_h\cdot n}
        + 2\av{u\cdot n}\av{v_h\cdot n}} \dif s 
      \\    
      & + \int_{\Gamma_c} p_f \jump{v_h\cdot n} \dif s .         
    \end{split}
  \end{equation}
  Combining
  \cref{eq:ab-form-consistency,eq:Gamma_c_identity,eq:Gamma_s_identity},
  we obtain $a_h({u}, {v}_h) + b_h({v}_h, \boldsymbol{p}) = 0$.

  To show $R=0$ note that
  \begin{align}
    \label{eq:bc-form-consistency}
    -b_h(u, \boldsymbol{q}_h) + c_h(\boldsymbol{p}, \boldsymbol{q}_h)
    = - \sum_{F \in \mathcal{F}_c^f} \int_F \jump{u \cdot n} q_{h,f} \dif s
    + \sum_{K\in\mathcal{T}} \int_K g q \dif x + c_h^f(p_f, q_{h,f})
  \end{align}
  because $u\cdot n$ is single-valued on
  $\mathcal{F}_0 \cup \mathcal{F}_s^f$ by \cref{eq:fault_problem_g},
  and $\nabla \cdot u = g$ by \cref{eq:fault_problem_b}.  We also
  multiply \cref{eq:fault_problem_c} by a test function
  $q_{h,f} \in Q_h^f$, integrate over faces $F \in \mathcal{F}_c^f$,
  and sum over the faces to obtain
  \begin{equation*}
    -\sum_{F \in \mathcal{F}_c^f} \int_F \nabla_{\tau}\cdot (\kappa_f \nabla_{\tau} p_f) q_{h,f} \dif s
    = \sum_{F \in \mathcal{F}_c^f} \int_F g_f q_{h,f} \dif s
    + \sum_{F \in \mathcal{F}_c^f} \int_F \jump{u \cdot n} q_{h,f} \dif s.  
  \end{equation*}
  Recall the definition of $\tilde{h}_e$ in \eqref{eq:htilde-def}.
  Integrating by parts and using the symmetric interior penalty DG
  method (see, for example, \cite{Riviere:book}) we obtain:
  \begin{equation}
    \label{eq:c_f-consistency}
    \begin{split}
      c_h^f(p_f, q_{h,f})
      &= \sum_{F \in \mathcal{F}_c^f} \int_F \kappa_f \nabla_{\tau} p_f \cdot \nabla_{\tau} q_{h,f} \dif s
      + \sum_{e \in \mathcal{E}_c^{f,int,D}} \int_e \frac{\sigma}{\tilde{h}_e} \av{\kappa_f}\jump{p_f}\jump{q_{h,f}} \dif l
      \\
      & - \sum_{e \in \mathcal{E}_c^{f,int,D}} \int_e \av{\kappa_f\nabla_{\tau}p_f} \cdot \jump{q_{h,f} n_e} \dif l
      - \sum_{e \in \mathcal{E}_c^{f,int,D}} \int_e \av{\kappa_f\nabla_{\tau}q_{h,f}} \cdot \jump{p_fn_e} \dif l
      \\
      =& \sum_{F \in \mathcal{F}_c^f} \int_F g_f q_{h,f} \dif s
      + \sum_{F \in \mathcal{F}_c^f} \int_F \jump{u \cdot n} q_{h,f} \dif s
      - \sum_{e \in \mathcal{E}_c^{f,N}} \int_e q_{h,f}g_f^N \cdot n_e \dif l
      \\
      & + \sum_{e \in \mathcal{E}_c^{f,D}} \int_e (\frac{\sigma}{\tilde{h}_e} \kappa_f p_f^D q_{h,f}
      - \kappa_f\nabla_{\tau}q_{h,f}\cdot n_e p_f^D)\dif l.      
    \end{split}
  \end{equation}
  Combining \cref{eq:bc-form-consistency} and
  \cref{eq:c_f-consistency}, we conclude that $R=0$.
\end{proof}

%------------------------------------------------------------------------------
\section{Well-posedness of the discrete problem}
\label{s:wellposedness}

In this section we assume that $p^D=0$ on $\partial \Omega$ for
simplicity of presentation. For general boundary condition $p^D \ne 0$
we can recast \cref{eq:discrete_fault} to an equivalent problem
finding $(u_h, \boldsymbol{p}_h^0) \in V_h \times \boldsymbol{Q}_h(0)$
with the same bilinear forms but with modified right-hand sides. To
simplify notation in this section, we will write $\boldsymbol{Q}_h$
instead of $\boldsymbol{Q}_h(0)$.

For the analysis in this and the following sections we define the
following norms:
\begin{align*}
  \norm[0]{{v}}_{{V}_h}^2
  &:= \norm[0]{v}_{\Omega}^2 
    + \sum_{F\in \mathcal{F}_c^f}\del[2]{\norm[0]{\jump{v \cdot n}}_F^2 + \norm[0]{\av{v\cdot n}}_F^2}
    + \sum_{F\in \mathcal{F}_s^f}\norm[0]{\av{v\cdot n}}_F^2
  && \forall \boldsymbol{v} \in {V}_h,
  \\
  \norm[0]{q_{f}}_{{Q}_{h,f}}^2
  &:= \sum_{F \in \mathcal{F}_c^f} \norm[0]{\nabla_{\tau} q_{f}}^2_{F}
    + \sum_{e \in \mathcal{E}_c^{f,int,D}} \frac{1}{\tilde{h}_e} \norm[0]{\jump{q_{f}}}_e^2
  && \forall q_f \in Q_h^f + H^1(\Gamma_c).
\end{align*}

\begin{lemma}[Coercivity and boundedness of $a_h$]
  The bilinear form $a_h$ is bounded and coercive on $V_h$, that is,
  \begin{subequations}
    \begin{align}
      \label{eq:ah-coercive}
      a_h({v}_h, {v}_h)
      &\ge \min \cbr[0]{\kappa_{\max}^{-1},\alpha_f^{-1}(\xi-\tfrac{1}{2})} \norm[0]{{v}_h}_{{V}_h}^2
      && \forall {v}_h \in {V}_h,
      \\ 
      \label{eq:ah-bounded}
      a_h({u}_h,{v}_h)
      &\le \max\cbr[0]{\kappa_{\min}^{-1}, 2\alpha_f^{-1}} \norm[0]{{u}_h}_{{V}_h} \norm[0]{{v}_h}_{{V}_h}
      && \forall {u}_h,{v}_h \in {V}_h.
    \end{align}    
  \end{subequations}
\end{lemma}
\begin{proof}
  \Cref{eq:ah-coercive} is straightforward by the definitions of $a_h$
  and $\norm[0]{\cdot}_{V_h}$. The proof of \cref{eq:ah-bounded} in
  addition uses the Cauchy--Schwarz inequality and the assumption that
  $\xi \leq 1$.
\end{proof}

\begin{lemma}[$c_h$ is positive semi-definite]
  \label{lem:ch-coercive}
  For sufficiently large  penalty parameter $\sigma>0$,
  \begin{equation*}
    c_h(\boldsymbol{q}_h,\boldsymbol{q}_h)
    \ge C \kappa_{\min} \norm[0]{q_{h,f}}_{Q_{h,f}}^2
    + \alpha \sum_{K \in \mathcal{T}} \norm[0]{ q_{h}-\bar{q}_{h} }_{\partial K \setminus (\Gamma_c \cup \Gamma_s)}^2
    \quad \forall \boldsymbol{q}_h \in \boldsymbol{Q}_h,
  \end{equation*}
  where $C$ depends on a discrete trace inequality constant and the
  penalty parameter $\sigma$.
\end{lemma}
\begin{proof}
  First note that for sufficiently large $\sigma>0$ we
  have (see, for example, \cite[Lemma 4.12]{Pietro:book}),
  \begin{equation*}
    c_h^f(q_{h,f}, q_{h,f}) \ge C\kappa_{\min} \norm[0]{q_{h,f}}_{Q_{h,f}}^2.
  \end{equation*}
  The result follows noting that
  \begin{equation*}
    c_h(\boldsymbol{q}_h,\boldsymbol{q}_h)
    = c_h^f(q_{h,f}, q_{h,f})
    + \sum_{K \in \mathcal{T}} \int_{\partial K\setminus (\Gamma_c \cup \Gamma_s)} \alpha (q_{h}-\bar{q}_{h})^2 \dif s.
  \end{equation*}
\end{proof}

The following result, proven in \cite[Proposition~2.1]{Cockburn:2008},
will be used to prove \Cref{thm:wellposedness}, i.e., the
well-posedness of \cref{eq:discrete_fault}.

\begin{lemma}
  \label{lem:Pitilde}
  Suppose that $w \in [H^1(K)]^{\rm dim}$ for a simplex
  $K \in \mathcal{T}$. For a fixed face $F$ of $K$, there exists a
  unique linear interpolation
  $\tilde{\Pi}_K^F:[H^1(K)]^{\rm dim} \rightarrow [P_k(K)]^{\rm dim}$
  such that
  \begin{subequations}
    \label{eq:Pitilde-def-eqs}
    \begin{align}
      \label{eq:Pitilde-def-eq1}
      \int_K (\tilde{\Pi}^F_K w - w) \cdot v \dif x &= 0,
      & & \forall v \in [P_{k-1}(K)]^{\rm dim},\ k \ge 1, 
      \\
      \label{eq:Pitilde-def-eq2}
      \int_{F'}  (\tilde{\Pi}^F_K w - w) \cdot n r \dif s &= 0,
      & & \forall r \in P_k(F'),\ \forall \text{ face } F' \subset \partial K,\ F' \ne F. 
    \end{align}
  \end{subequations}
  Moreover, for the same $F$ and $\tilde{\Pi}^F_K$, if
  $w \in [H^{s+1}(K)]^{\rm dim}$, $0 \le s \le k$, then
  \begin{subequations}        
    \label{eq:Pitilde-approx-ineqs}
    \begin{align}
      \label{eq:Pitilde-approx-ineq1}
      \norm[0]{ \tilde{\Pi}_K^F w \cdot n - P_F w \cdot n }_F & \le C h_K^{s+1/2} \norm[0]{ P_K \nabla \cdot w }_{H^s(K)}, 
      \\
      \label{eq:Pitilde-approx-ineq2}
      \norm[0]{ \tilde{\Pi}_K^F w - w }_K & \le C h_K^{s+1} \norm[0]{ w }_{H^{s+1}(K)},
    \end{align}
  \end{subequations}
  where $P_F$ is the $L^2$-projection into $P_k(F)$ and $P_K$ is the
  $L^2$-projection into $P_k(K)$.
\end{lemma}

\begin{theorem}[Well-posedness]
  \label{thm:wellposedness}
  Suppose that each simplex $K \in \mathcal{T}$ has a face in
  $\mathcal{F}_0$ and each connected component of
  $\Omega \setminus \Gamma_c$ has a part of its boundary intersecting
  with $\partial \Omega$ with positive $(\dim-1)$-dimensional Lebesgue
  measure.  For sufficiently large penalty parameter $\sigma$, the
  discrete problem \cref{eq:discrete_fault} is well-posed.
\end{theorem}
\begin{proof}
  To show well-posedness of the finite-dimensional linear problem
  \cref{eq:discrete_fault}, it is sufficient to show uniqueness. For
  this, let $g_f=0$ in $\Gamma_c$, $g=0$ in $\mathring{\Omega}$,
  $p_f^D=0$ on $(\partial \Gamma_c)^D$, and $g_f^N=0$ on
  $(\partial \Gamma_c)^N$, choose ${v}_h={u}_h$ and
  $\boldsymbol{q}_h=\boldsymbol{p}_h$ in \cref{eq:discrete_fault}, and
  add \cref{eq:discrete_fault_a,eq:discrete_fault_b} to find
  $a_h({u}_h,{u}_h) + c_h(\boldsymbol{p}_h,\boldsymbol{p}_h) = 0$. The
  coercivity
  $a_h(u_h,u_h) \ge \kappa_{\max}^{-1} \norm[0]{u_h}_{\Omega}^2$ by
  \cref{eq:ah-coercive} and positive semi-definiteness of $c_h$ by
  \cref{lem:ch-coercive} immediately imply that $u_h =0$ in $\Omega$,
  that $p_{h,f}=0$ in $\Gamma_c$, that $p_{h} = \bar{p}_{h}$ for all
  $F \in \mathcal{F}_0$ and since we assumed $p^D=0$, that
  $p_h = \bar{p}_h = 0$ for all $F \in \mathcal{F}_b$. We are
  therefore left to show that $p_{h} = 0$ in $\Omega$ and
  $\bar{p}_{h} = 0$ on all $F \in \mathcal{F}_0 \cup \mathcal{F}_s^f$.
  To show this, first observe that since $u_h = 0$ and $p_{h,f} = 0$
  after integrating by parts in \cref{eq:discrete_fault_a}, we obtain
  \begin{equation}
    \label{eq:intermediate}
    \sum_{K\in \mathcal{T}} \int_K \nabla p_h  \cdot v_h  \dif x
    -\sum_{K\in \mathcal{T}} \int_{\partial K} p_h   v_h\cdot n  \dif s
    + \sum_{K\in \mathcal{T}} \int_{\partial K \backslash \Gamma_c} \bar{p}_h v_h \cdot n \dif s = 0 
    \quad \forall {v}_h \in {V}_h.
  \end{equation}
  Next, we let $K \in \mathcal{T}$, and let $F \subset \partial K$ be
  a face such that $F \in \mathcal{F}_0$. By the existence of
  $\tilde{\Pi}_K^F$ in \cref{lem:Pitilde}, there exists
  $v_h^K \in [P_k(K)]^{\dim}$ such that $v_h^K \cdot n = 0$ on any
  face $F' \subset \partial K$, $F' \ne F$, and
  \begin{equation*}
    \int_K v_h^K \cdot w_h \dif x = \int_K \nabla p_{h} \cdot w_h \dif x
    \quad \forall w_h \in [P_{k-1}(K)]^{\dim}.
  \end{equation*}
  Given $K \in \mathcal{T}$ choose $v_h\in V_h$ such that
  \begin{equation}
    \label{eq:well-posed-testfunction}
    v_h = 
    \begin{cases}
      v_h^K & \text{on } K,
      \\
      0 & \text{otherwise}.
    \end{cases}
  \end{equation}
  Choosing $v_h$ in \cref{eq:intermediate} as described by
  \cref{eq:well-posed-testfunction} and using that
  $p_{h} = \bar{p}_{h}$ for all
  $F \in \mathcal{F}_0 \cup \mathcal{F}_b$, we find
  \begin{equation}
    \label{eq:grad_p-vanish}
    b_h(v_h, \boldsymbol{p}_h)
    = \int_K \nabla p_{h} \cdot \nabla p_{h} \dif x = 0. 
  \end{equation}
  Therefore, $\nabla p_h=0$ on $K$ which implies that $p_{h}$ is a
  constant on $K$. Since $K \in \mathcal{T}$ is arbitrary, $p_h$ is
  element-wise constant on $\Omega$.
  
  Finally, let $F_s \in \mathcal{F}_s^f$. Then,
  $F_s = \partial K_1 \cap \partial K_2$ for some
  $K_1,K_2\in \mathcal{T}$. This time we define $v_h \in V_h$ such
  that it vanishes on $\Omega \setminus (K_1 \cup K_2)$ and such that
  its normal component takes the following facial values as a function
  on $K_i$:
  \begin{equation}
    \label{eq:well-posed-testfunction2}
    v_h\cdot n|_{F} = 
    \begin{cases}
      -p_h|_{K_i} + \bar{p}_h \quad &\text{ if } F = F_s
      \\
      0 \quad &\text{ if } F \not = F_s        
    \end{cases} \qquad \text{ for }i=1,2,
  \end{equation}
  and all other local BDM degrees of freedom vanish. Using this $v_h$
  in \cref{eq:intermediate} together with the fact that
  $\nabla p_h = 0$ element-wise on $\Omega$, we obtain that
  $p_h = \bar{p}_h$ on $F_s$. Since $F_s \in \mathcal{F}_s^f$ is
  arbitrary, we conclude that $p_h = \bar{p}_h$ on
  $\Gamma_s$. Furthermore, since $p_h$ is constant on every
  $K \in \mathcal{T}$, using that $p_{h}=\bar{p}_{h}$ on all faces in
  $\mathcal{F}_0 \cup \mathcal{F}_s^f$ we conclude that $p_{h}$ and
  $\bar{p}_h$ are constant on each connected component of
  $\Omega \setminus \Gamma_c$. Recalling that $\bar{p}_h=0$ on
  $\mathcal{F}_b$, $p_h$ and $\bar{p}_h$ must also vanish and so
  \cref{eq:discrete_fault} has a unique solution.
\end{proof}
\begin{remark}
  We assume that every connected component of
  $\Omega \setminus \Gamma_c$ has a boundary part intersecting with
  $\partial \Omega$. If this is not true, then there are connected
  components whose boundaries are included in $\Gamma_c$. For
  simplicity of presentation, assume that there is only one such
  connected component $\Omega_0$. Then, the compatibility condition
  $\int_{\partial \Omega_0} u \cdot n \dif s = \int_{\Omega_0} \nabla
  \cdot u \dif x$ is necessary as in pure Neumann boundary condition
  problems. Its discretization needs a modified bilinear form with
  Lagrange multiplier $q_0 \in \mathbb{R}$ where
  $\tilde{b}_h(v_h, \tilde{\boldsymbol{q}}_h) =
  b_h(v_h,\boldsymbol{q}_h) + q_0 (\int_{\Omega_0} \nabla \cdot v_h
  \dif x - \int_{\partial \Omega_0} v_h \cdot n \dif s)$ and
  $\tilde{\boldsymbol{q}}_h = (q_h, \bar{q}_h, q_{h,f}, q_{h,0}) \in
  \boldsymbol{Q}_h \times \mathbb{R}$. For well-posedness we can
  obtain $u_h=0$, $p_{h,f}=0$, and $\bar{p}_h = 0$ on $\mathcal{F}_0$
  by the same argument. We can also obtain that $p_h$ is element-wise
  constant by the same argument as in \cref{eq:grad_p-vanish} by using
  the $v_h$ defined in \cref{eq:well-posed-testfunction}. On the
  connected component $\Omega \setminus \overline{\Omega_0}$, we can
  show $p_h=0$, $\bar{p}_h=0$ as in the proof
  of~\cref{thm:wellposedness}, so it suffices to show
  $p_h=\bar{p}_h=p_{h,0}=0$ on $\Omega_0$. For
  $F \in \mathcal{F}_0 \cap \Omega_0$, by taking the test function as
  in \cref{eq:well-posed-testfunction2} with the right-hand side
  quantity replaced by $p_h|_{K_i} + p_{h,0}$, we can show that
  $p_h = - p_{h,0}$ on $\Omega_0$ because we assume that every
  $K\in \mathcal{T}$ contains a face in $\mathcal{F}_0$. Then,
  $\tilde{b}_h(v_h, \tilde{\boldsymbol{p}}_h)=0$ is reduced to
  $\sum_{K\in \mathcal{T}, K \subset \Omega_0} \int_{\partial K
    \cap\Gamma_s} \bar{p}_h v_h\cdot n \dif s + p_{h,0} \int_{\partial
    \Omega_0} v_h\cdot n \dif s = 0$. Taking $v_h$ as an arbitrary
  constant vector on $\Omega_0$ implies $p_{h,0}=p_h=0$ and taking the
  test function as in \cref{eq:well-posed-testfunction2} implies
  $\bar{p}_h=0$.
\end{remark}

%------------------------------------------------------------------------------
\section{Error Analysis}
\label{s:erroranalysis}

In this section we prove a priori error estimates for the solution of
our HDG/DG discretization defined by \cref{eq:discrete_fault}.

We begin with defining an interpolation operator.  On $Q_h$ and $V_h$
we consider the HDG projection
$(\Pi v,\Pi q)\in [P_k(K)]^{\rm dim}\times P_k(K)$. This projection is
defined on any $K\in \mathcal{T}$ and for any
$(v, q) \in [H^1(K)]^{\rm dim} \times H^1(K)$ as follows
(cf. \cite[(3.4)]{Sayas:book}):
\begin{subequations}
  \label{eq:HDG-projection-eqs}
  \begin{align}
    \label{eq:HDG-projection-eq1}
    \int_K \Pi v \cdot \tilde{v}\dif x
    &=\int_K v \cdot \tilde{v} \dif x & & \forall \tilde{v} \in [P_{k-1}(K)]^{\rm dim},
    \\
    \label{eq:HDG-projection-eq2}
    \int_K \Pi q\, \tilde{q} \dif x
    &= \int_K q \,\tilde{q} \dif x & & \forall \tilde{q} \in P_{k-1}(K),
    \\
    \label{eq:HDG-projection-eq3}
    \int_F (\Pi v \cdot n + \tilde{\alpha} \Pi q) r \dif s
    &=\int_F (v \cdot n +\tilde{\alpha} q ) r \dif s  & & \forall r \in P_k(F),\ F \subset \partial K,
  \end{align}
\end{subequations}
where $\tilde{\alpha} = \alpha$ if
$F \subset \mathcal{F}_0 \cup \mathcal{F}_b$, and $\tilde{\alpha} = 0$
if $F \in \mathcal{F}_c^f \cup \mathcal{F}_s^f$. Note that $\Pi v$ and
$\Pi q$ both depend on $v$ and $q$, however, for notational simplicity
we do not write down this dependence. Furthermore, let $\Bar{\Pi}_{Q}$
be the $L^2$-projection operator onto $\bar{Q}_h$, and let
$\Pi_{Q^f}:H^s(\Gamma_c)\mapsto Q_h^f$, $s >3/2$ be the following
elliptic projection:
\begin{equation}
  \label{eq:Qf-elliptic-projection}
  c_h^f(p^f,q_h^f)=c_h^f(\Pi_{Q^f} p^f,q_h^f) \qquad \forall q_h^f \in Q_h^f.
\end{equation}
We introduce the following decomposition of the errors:
\begin{equation}
  \label{eq:error-terms}
  \begin{aligned}
    e^I_{u}&=u-\Pi u,
    & e^h_{u}&=u_{h}-\Pi u, 
    \\
    e^I_{p}&=p-\Pi p,
    & e^h_{p}&=p_{h}-\Pi p, 
    \\
    \bar{e}^I_{p} & = \av{p}-\bar{\Pi}_Q \av{p},
    & \bar{e}^h_{p} & =\bar{p}_{h}-\bar{\Pi}_{Q} \av{p},
    \\
    e^I_{p_f}&=p_f-\Pi_{Q^f} p_f,
    & e^h_{p_f}&=p_{h,f}-\Pi_{Q^f} p_f.
  \end{aligned}
\end{equation}
We will use the following compact notation:
$\boldsymbol{e}^h_p=(e^h_{p},\Bar{e}^h_{p},e^h_{p_f})$, and
$\boldsymbol{e}^I_p=(e^I_{p},\Bar{e}^I_{p},e^I_{p_f})$.

A standard error estimate of interior penalty discontinuous Galerkin
methods (see e.g., \cite{Arnold:2002,Arnold:1982}) and a discrete
Poincar\'{e} inequality (cf. \cite{Brenner:2003}) give
\begin{subequations}
  \begin{align}
    \label{eq:pf-interpolation-H1-error}
    \norm[0]{ e_{p_f}^I }_{Q_{h,f}} &\le Ch^s \norm[0]{ p_f }_{H^{s+1}(\Gamma_c)}, && \tfrac{1}{2} < s \le k_f,
    \\
    \label{eq:pf-interpolation-L2-error}
    \norm[0]{ e_{p_f}^I }_{\Gamma_c} &\le Ch^s \norm[0]{ p_f }_{H^{s+1}(\Gamma_c)}, && \tfrac{1}{2} < s \le k_f.
  \end{align}  
\end{subequations}
If $\Gamma_c$ is convex, a standard duality argument results in
\begin{equation}
  \label{eq:pf-interpolation-duality-error}
  \norm[0]{ e_{p_f}^I }_{\Gamma_c} \le Ch^{s+1} \norm[0]{ p_f }_{H^{s+1}(\Gamma_c)}, \quad \tfrac{1}{2} < s \le k_f +1.
\end{equation}
Assuming that
$(u, p) \in [H^{s}(\mathring{\Omega})]^{\rm dim} \times
H^{s}(\mathring{\Omega})$ for $1 \le s \le k+1$, we have by
\cite[Proposition~3.5]{Sayas:book} that
\begin{subequations} 
  \label{eq:HDG-projection-estimates}
  \begin{align}
    \label{eq:HDG-projection-p-estimate}
    \norm[0]{ e_{p}^I }_{\Omega}
    &\le C h^{s} (\norm[0]{ p }_{H^{s}(\mathring{\Omega})} + \alpha^{-1} \norm[0]{ \nabla \cdot u }_{H^{s-1}(\mathring{\Omega})}),
    \\
    \label{eq:HDG-projection-u-estimate}
    \norm[0]{ e_{u}^I }_{\Omega}
    &\le Ch^{s}(\norm[0]{ u }_{H^{s}(\mathring{\Omega})} + \alpha \norm[0]{p}_{H^{s}(\mathring{\Omega})}). 
    \end{align}
\end{subequations}
We obtain the error equations in the following lemma.
\begin{lemma}[Error equation]
  For any $({v}_h,\boldsymbol{q}_h)\in {V}_h\times\boldsymbol{Q}_h$
  the following system of equations holds:
  \begin{subequations}
    \label{eq:Error_eq}
    \begin{align}
      \label{eq:Error_eq_a}
      a_h({e}^h_u,{v}_h) + b_h({v}_h,\boldsymbol{e}_p^h)
      &=\sum_{K\in \mathcal{T}} \int_K \kappa^{-1} e_{u}^I\cdot v_{h} \dif x
        + \sum_{F\in \mathcal{F}_c^f} \int_{F} e_{p_f}^I \jump{v_h \cdot n} \dif s,
      \\   
      \label{eq:Error_eq_b}
      -b_h({e}^h_u,\boldsymbol{q}_h) + c_h(\boldsymbol{e}^h_p,\boldsymbol{q}_h)
      &=0.
    \end{align}
  \end{subequations}
\end{lemma}
\begin{proof}
  Let $(u,p,p_f)$ be the solution of
  \cref{eq:fault_problem,eq:fault_problem_bcs} and write
  $\boldsymbol{p}=(p,\av{p},p_f)$.  Furthermore, let
  $(u_h,\boldsymbol{p}_h)$ be the solution of
  \cref{eq:discrete_fault}. Then, by consistency (see
  \cref{lem:consistency}), we obtain:
  \begin{subequations}
    \label{eq:error-eqs}
    \begin{align}
      \label{eq:error-eq1}
      a_h(u - u_h, {v}_h)
      + b_h({v}_h,\boldsymbol{p} - \boldsymbol{p}_h)
      &= 0 & & \forall {v}_h \in {V}_h, 
      \\
      \label{eq:error-eq2}
      - b_h({u}-{u}_h, \boldsymbol{q}_h)
      + c_h(\boldsymbol{p}-\boldsymbol{p}_h,\boldsymbol{q}_h)
      &= 0 & & \forall \boldsymbol{q}_h \in \boldsymbol{Q}_h.
    \end{align}
  \end{subequations}
  Using the decomposition of the errors introduced in
  \cref{eq:error-terms},
  \begin{align*}
    a_h({e}^h_u,{v}_h)
    + b_h({v}_h,\boldsymbol{e}_p^h)
    & = a_h({e}^I_u,{v}_h)
      + b_h({v}_h,\boldsymbol{e}_p^I)  & & \forall {v}_h \in {V}_h,
    \\
    -b_h({e}^h_u,\boldsymbol{q}_h)
    + c_h(\boldsymbol{e}^h_p,\boldsymbol{q}_h)
    & = - b_h({e}^I_u,\boldsymbol{q}_h)
      + c_h(\boldsymbol{e}^I_p,\boldsymbol{q}_h)  & & \forall \boldsymbol{q}_h \in \boldsymbol{Q}_h. 
  \end{align*}      
  By the definitions of $a_h(\cdot, \cdot)$ in \cref{eq:ah} and
  $b_h(\cdot, \cdot)$ in \cref{eq:bh}, and by properties of the HDG
  projection \cref{eq:HDG-projection-eqs}, we obtain
  \begin{equation*}
    \begin{split}
      a_h({e}_u^I,{v}_h)
      =& \sum_{K\in \mathcal{T}} \int_K \kappa^{-1} e_{u}^I\cdot v_{h} \dif x 
      +\sum_{F\in\mathcal{F}_s^f}\int_F 2\alpha_f^{-1}\av{e_u^I\cdot n}\av{v_h\cdot n}\dif s
      \\
      & + \sum_{F_f\in \mathcal{F}_c^f} \int_{F}\alpha_f^{-1}\big((\xi-\tfrac{1}{2})\jump{e_{u}^I\cdot n} \jump{v_h\cdot n}
      + 2\av{e_{u}^I\cdot n} \av{v_h\cdot n} )\dif s
      \\
      =& \sum_{K\in \mathcal{T}} \int_K \kappa^{-1} e_{u}^I\cdot v_{h} \dif x,      
    \end{split}
  \end{equation*}
  and
  \begin{equation*}
    \begin{split}
      b_h({v}_h,\boldsymbol{e}_p^I)
      &= -\sum_{K\in \mathcal{T}} \int_K e_{p}^I \nabla \cdot v_{h} \dif x
      + \sum_{K\in \mathcal{T}} \int_{\partial K\backslash \Gamma_c} \bar{e}_{p}^I v_{h} \cdot n \dif s
      + \sum_{F\in \mathcal{F}_c^f} \int_{F} e_{p_f}^I \jump{v_h \cdot n} \dif s
      \\
      &= \sum_{F\in \mathcal{F}_c^f} \int_{F} e_{p_f}^I \jump{v_h \cdot n} \dif s.      
    \end{split}
  \end{equation*}
  These identities complete the proof of \cref{eq:Error_eq_a}.

  For \cref{eq:Error_eq_b}, note that integration by parts gives
  \begin{align*}
    & -b_h({e}_{u}^I,\boldsymbol{q}_h)
      = \sum_{K\in \mathcal{T}} \int_K q_{h} \nabla \cdot e_{u}^I \dif x
      - \sum_{K\in \mathcal{T}} \int_{\partial K\backslash \Gamma_c} \bar{q}_{h} e_{u}^I \cdot n \dif s
      - \sum_{F_f\in \mathcal{F}_c^f} \int_{F} q_{h,f}\jump{e_{u}^I \cdot n} \dif s
    \\
    &= -\sum_{K\in \mathcal{T}} \int_K \nabla q_{h} \cdot e_{u}^I \dif x
      + \sum_{K\in \mathcal{T}} \int_{\partial K\backslash \Gamma_c} (q_{h}-\bar{q}_{h}) e_{u}^I \cdot n \dif s
      - \sum_{F\in \mathcal{F}_c^f} \int_{F} q_{h,f}\jump{e_{u}^I \cdot n} \dif s 
    \\
    &\quad + \sum_{K \in \mathcal{T}} \int_{\partial K \cap \Gamma_c} q_{h} {e_{u}^I \cdot n} \dif s
    \\
    &= \sum_{K\in \mathcal{T}} \int_{\partial K\backslash \Gamma_c} (q_{h}-\bar{q}_{h}) e_{u}^I \cdot n \dif s
      - \sum_{F\in \mathcal{F}_c^f} \int_{F} q_{h,f}\jump{e_{u}^I \cdot n} \dif s
      + \sum_{K \in \mathcal{T}} \int_{\partial K \cap F, F \in \mathcal{F}_c^f} q_{h} {e_{u}^I \cdot n} \dif s
    \\
    &= \sum_{K\in \mathcal{T}} \int_{\partial K\backslash \Gamma_c} (q_{h}-\bar{q}_{h}) e_{u}^I \cdot n \dif s 
    \\
    &= \sum_{K\in \mathcal{T}} \int_{\partial K\backslash (\Gamma_c \cup \Gamma_s)} (q_{h}-\bar{q}_{h}) e_{u}^I \cdot n \dif s,
  \end{align*}
  where the last two equalities are because of
  \cref{eq:HDG-projection-eq3}. Furthermore, by definition of
  $c_h(\cdot, \cdot)$ in \cref{eq:ch},
  \cref{eq:Qf-elliptic-projection}, and the definition of
  $\bar{\Pi}_{Q}$,
  \begin{equation}
  \label{eq:chepI}
      \begin{split}
    c_h(\boldsymbol{e}_p^I,\boldsymbol{q}_h)
    &= c_h^f (e_{p_f}^I, q_{h,f})
      + \sum_{K\in \mathcal{T}}\int_{\partial K \setminus (\Gamma_c \cup \Gamma_s)} \alpha (e^I_{p}-\bar{e}^I_{p})(q_{h}-\bar{q}_{h})\dif s
    \\
    &= \sum_{K\in \mathcal{T}}\int_{\partial K \setminus (\Gamma_c\cup \Gamma_s)} \alpha (e^I_{p}-\bar{e}^I_{p})(q_{h}-\bar{q}_{h})\dif s
    \\
    & =\sum_{K\in \mathcal{T}}\int_{\partial K \setminus (\Gamma_c\cup \Gamma_s)} \alpha e^I_{p}(q_{h}-\bar{q}_{h})\dif s.          
      \end{split}
  \end{equation}
  By \cref{eq:HDG-projection-eq3} we find
  $-b_h(e_{u}^I,\boldsymbol{q}_h)+c_h(\boldsymbol{e}_p^I,\boldsymbol{q}_h)
  = 0$, which completes the proof of \cref{eq:Error_eq_b}.
\end{proof}

\begin{theorem}
  \label{thm:errorestimates}
  Suppose that $u \in [H^s(\mathring{\Omega})]^{\rm dim}$,
  $p \in H^s(\mathring{\Omega})$, $1\le s \le k+1$ and
  $p_f \in H^{s_f+1}(\Gamma_c)$ for $\tfrac{1}{2} < s_f \le
  k_f$. Then,
  \begin{multline}
    \label{eq:u-pf-estimate}
    \norm[0]{ u - u_{h} }_{\Omega}  + \norm[0]{ p_f - p_{f,h} }_{Q_{h,f}} 
    \le C \kappa_{\min}^{-1} h^s (\norm[0]{ u}_{H^s(\mathring{\Omega})} + \norm[0]{ p }_{H^s(\mathring{\Omega})} )
    + C h^{s_f} \norm[0]{ p_f }_{H^{s_f+1}(\Gamma_c)},
  \end{multline}
  with $C>0$ depending on a discrete trace inequality constant, the
  penalty parameter $\sigma$, and the parameters $\kappa_{\min}$,
  $\alpha_f^{-1}$, and $\xi$, and
  \begin{equation}
    \label{eq:p-estimate}    
      \norm[0]{ p - p_{h} }_{\Omega}
      \le C_1 h^s (\norm[0]{ u }_{H^s(\mathring{\Omega})} + \norm[0]{ p }_{H^s(\mathring{\Omega})} )
      + C_2 h^{s_f} \norm[0]{ p_f }_{H^{s_f+1}(\Gamma_c)},
  \end{equation}
  where $C_1,C_2>0$ depend on $\kappa_{\min}$, $\alpha_f^{-1}$, and
  $\xi$.
  
  Moreover, if \cref{eq:pf-interpolation-duality-error} holds, then
  $h^{s_f}$ in the above estimates can be replaced by $h^{s_f+1}$.
\end{theorem}
\begin{proof}
  Taking ${v}_h = {e}_u^h$, $\boldsymbol{q}_h = \boldsymbol{e}_p^h$ in
  \cref{eq:Error_eq}, adding \cref{eq:Error_eq_a,eq:Error_eq_b}, and
  using the Cauchy--Schwarz inequality, we obtain:
  \begin{equation}
    \label{eq:error-boundedness}
    a_h({e}^h_u, {e}^h_u)+c_h(\boldsymbol{e}^h_p,\boldsymbol{e}^h_p) \le
    \del[3]{\kappa_{\min}^{-1}\norm[0]{e_{u}^I}_{\Omega} + \norm[0]{e_{p_f}^I}_{\Gamma_c} }\norm[0]{{e}_u^h}_{V_h}.
  \end{equation}
  Using the coercivity of $a_h$ (see \cref{eq:ah-coercive}) and the
  positive semi-definiteness of $c_h$ (see \cref{lem:ch-coercive}), we
  obtain:
  \begin{equation}
    \label{eq:error-coercivity}
    a_h({e}^h_u,{e}^h_u) + c_h(\boldsymbol{e}^h_p,\boldsymbol{e}^h_p)
    \ge
    C_c \del[1]{ \norm[0]{{e}^h_u}_{{V}_h}^2 + \norm[0]{e_{p_f}^h}_{{Q}_{h,f}}^2 }
    + \alpha \sum_{K \in \mathcal{T}} \norm[0]{ e_{p}^h-\bar{e}_{p}^h }_{\partial K \setminus (\Gamma_c \cup \Gamma_s)}^2,
  \end{equation}
  with
  $C_c = \min \cbr[0]{\kappa_{\max}^{-1}, \alpha_f^{-1}(\xi-1/2),
    C\kappa_{\min}}$. Applying Young's inequality to the right hand
  side of \cref{eq:error-boundedness} and combining
  \cref{eq:error-boundedness,eq:error-coercivity}, we obtain:
  \begin{multline}
    \label{eq:u-pf-h-estimate}
    \tfrac{1}{2}C_c \norm[0]{{e}^h_u}_{V_h}^2
    + C_c \norm[0]{e_{p_f}^h}_{{Q}_{h,f}}^2
    + \alpha \sum_{K \in \mathcal{T}} \norm[0]{e_{p}^h-\bar{e}_{p}^h}_{\partial K \setminus (\Gamma_c \cup \Gamma_s)}^2
    \le \tfrac{1}{2} C_c^{-1} \del[3]{
      \kappa_{\min}^{-1} \norm[0]{ e_{u}^I }_{\Omega} + \norm[0]{e_{p_f}^I}_{\Gamma_c} }^2.
  \end{multline}
  \Cref{eq:u-pf-estimate} follows by a triangle inequality,
  \cref{eq:pf-interpolation-L2-error,eq:HDG-projection-u-estimate}.

  We next prove \cref{eq:p-estimate}. It is known
  (cf. \cite{Girault-Raviart:book} and \cite[Lemma
  B.1]{Baerland:2017}) that there exists
  $w \in [H^1(\Omega)]^{\rm dim}$ such that
  $\nabla \cdot w = -e_{p}^h$ and
  $\norm[0]{w}_{H^1(\Omega)} \le C \norm[0]{e_{p}^h}_{\Omega}$ with
  $C>0$ only depending on $\Omega$. 
  Then,
  \begin{equation*}
    \norm[0]{e_{p}^h}_{\Omega}^2
    = -\int_{\Omega} e_{p}^h \nabla \cdot w \dif x
    = \int_{\Omega} \nabla e_{p}^h \cdot w \dif x
    - \sum_{K \in \mathcal{T}} \int_{\partial K}  e_{p}^h w\cdot n \dif s.
  \end{equation*}
  For every $K\in \mathcal{T}$, we fix a face $F\in \mathcal{F}_0$ of
  $K$ and define $\tilde{\Pi}_K^F$ using
  \cref{eq:Pitilde-def-eqs}. Then we introduce
  $\tilde{\Pi}:[H^1(\Omega)]^{\rm dim}\rightarrow V_h$ such that
  $\tilde{\Pi}w=\tilde{\Pi}_K^Fw$ on each $K$.  By
  \cref{eq:Pitilde-def-eq1},
  $\int_{\Omega} \nabla e_{p}^h \cdot w \dif x = \int_{\Omega} \nabla
  e_{p}^h \cdot \tilde{\Pi} w \dif x$. Applying this to the above
  identity,
  \begin{equation}
    \label{eq:normepih-1}
    \begin{split}
      \norm[0]{e_{p}^h}_{\Omega}^2
      =& \int_{\Omega} \nabla e_{p}^h \cdot \tilde{\Pi} w \dif x
      - \sum_{K \in \mathcal{T}} \int_{\partial K}  e_{p}^h w \cdot n \dif s 
      \\
      =& \int_{\Omega} \nabla e_{p}^h \cdot \tilde{\Pi} w \dif x
      + \sum_{K \in \mathcal{T}} \int_{\partial K \setminus (\Gamma_c \cup \Gamma_s)}  (\bar{e}_{p}^h - e_{p}^h) w \cdot n \dif s 
      \\
      & -\sum_{K \in \mathcal{T}} \int_{\partial K \cap (\Gamma_c \cup \Gamma_s)}  e_{p}^h w \cdot n \dif s,      
    \end{split}
  \end{equation}
  where for the second equality we used the normal continuity of
  $w\cdot n$ on $F \in \mathcal{F}_0 \cup \mathcal{F}_s^f$,
  single-valuedness of $\bar{e}_p^h$, and that $\bar{e}_p^h = 0$ on
  $F \in \mathcal{F}_b$. Using integration by parts we note that
  \begin{equation}
    \label{eq:normepih-2}
    \begin{split}
      b_h(\tilde{\Pi} w, \boldsymbol{e}_p^h)
      =& \int_{\Omega} \nabla e_{p}^h \cdot \tilde{\Pi} w \dif x
      + \sum_{K \in \mathcal{T}} \int_{\partial K \setminus (\Gamma_c \cup \Gamma_s)}  (\bar{e}_{p}^h - e_{p}^h) \tilde{\Pi} w \cdot n \dif s 
      \\
      & - \sum_{K \in \mathcal{T}} \int_{\partial K \cap (\Gamma_c \cup \Gamma_s)}  e_{p}^h\, \tilde{\Pi} w \cdot n \dif s
      + \sum_{F \in \mathcal{F}_c^f} \int_{F} e_{p_f}^h \jump{\tilde{\Pi} w \cdot n} \dif s.
    \end{split}
  \end{equation}
  Combining \cref{eq:normepih-1,eq:normepih-2},
  \begin{equation*}
    \begin{split}
      \norm[0]{e_{p}^h}_{\Omega}^2
      =& b_h(\tilde{{\Pi}} w, \boldsymbol{e}_p^h)
      + \sum_{K \in \mathcal{T}} \int_{\partial K \setminus (\Gamma_c \cup \Gamma_s)} (\bar{e}_{p}^h - e_{p}^h) (w - \tilde{\Pi} w) \cdot n \dif s 
      \\
      & - \sum_{K \in \mathcal{T}} \int_{\partial K \cap (\Gamma_c \cup \Gamma_s)}  e_{p}^h (w - \tilde{\Pi} w) \cdot n \dif s
      - \sum_{F \in \mathcal{F}_c^f} \int_{F} e_{p_f}^h \jump{\tilde{{\Pi}} w \cdot n} \dif s.
    \end{split}
  \end{equation*}
  Since $w \cdot n$ is continuous on $\Gamma_c$ and
  $\tilde{\Pi} w \cdot n |_F = P_F w \cdot n|_F$ for
  $F\in \mathcal{F}_c^f \cup \mathcal{F}_s^f$,
  \begin{equation}
    \label{eq:normepih-3}
    \norm[0]{e_{p}^h}_{\Omega}^2
    = b_h(\tilde{\Pi} w, \boldsymbol{e}_p^h)
    + \sum_{K \in \mathcal{T}} \int_{\partial K \setminus (\Gamma_c \cup \Gamma_s)}(\bar{e}_{p}^h - e_{p}^h) (w - \tilde{\Pi} w) \cdot n \dif s.
  \end{equation}
  Let $P_{\partial K} w \cdot n$ be the face-wise $L^2$-orthogonal
  projection of $w \cdot n$ into $P_k(\partial K)$. Taking
  ${v}_h = \tilde{\Pi} w$ in \cref{eq:Error_eq_a}, combining with
  \cref{eq:normepih-3}, using the $L^2$-orthogonality of
  $w\cdot n - P_{\partial K} w \cdot n$ to $P_k(\partial K)$, and the
  continuity of $\tilde{\Pi}w \cdot n$ on $\Gamma_c\cup \Gamma_s$, we
  obtain:
  \begin{equation}
    \label{eq:normepihI1I2I3}
    \begin{split}
      \norm[0]{e_{p}^h}_{\Omega}^2
      =& \sum_{K \in \mathcal{T}} \int_{\partial K \setminus (\Gamma_c\cup \Gamma_s)}  (\bar{e}_{p}^h - e_{p}^h) (P_{\partial K} w \cdot n - \tilde{\Pi} w \cdot n) \dif s - a_h({e}_u^h, \tilde{\Pi} w)
      \\
      & + \sum_{K\in \mathcal{T}} \int_K \kappa^{-1} e_{u}^I\cdot \tilde{\Pi} w \dif x 
      \\
      =&: I_1 + I_2 + I_3.
    \end{split}
  \end{equation}
  By the Cauchy--Schwarz inequality and
  \cref{eq:Pitilde-approx-ineq1},
  \begin{equation*}
    |I_1| \le C h^{\tfrac{1}{2}} \del[3]{ \sum_{K \in \mathcal{T}} \norm[0]{e_{p}^h-\bar{e}_{p}^h }_{\partial K \setminus (\Gamma_c\cup \Gamma_s)}^2
    }^{\tfrac{1}{2}} \norm[0]{e_{p}^h}_{\Omega}. 
  \end{equation*}
  We next bound $I_2$. By \cref{eq:ah-bounded},
  \begin{equation}
  \label{eq:firstboundI2}
      |I_2| \le 
      C\max\cbr[0]{\kappa_{\min}^{-1},\alpha_f^{-1}} \norm[0]{e_u^h}_{V_h} \norm[0]{\tilde{\Pi}w}_{V_h}.
  \end{equation}
  Using that $\tilde{\Pi}w\cdot n$ is continuous on $\Gamma_c$ in the
  definition of $\norm{\cdot}_{V_h}$:
  \begin{equation}
    \label{eq:checkingstuff}
    \norm[0]{\tilde{\Pi}w}_{{V}_h}^2
    = \norm[0]{\tilde{\Pi}w}_{\Omega}^2 
    + \sum_{F\in \mathcal{F}_c^f}\norm[0]{\av{\tilde{\Pi}w\cdot n}}_F^2
    + \sum_{F\in \mathcal{F}_s^f}\norm[0]{\av{\tilde{\Pi}w\cdot n}}_F^2.
  \end{equation}
  By the trace inequality
  $\norm[0]{ w }_{H^{1/2}(\Gamma_c\cup \Gamma_s)} \le C \norm[0]{ w
  }_{H^1(\Omega)}$ we find
  \begin{equation}
  \label{eq:pitildewneph}
    \norm[0]{ \av{\tilde{\Pi}w \cdot n} }_{\Gamma_c\cup \Gamma_s}
    \le 
    \norm[0]{ w \cdot n }_{\Gamma_c\cup \Gamma_s}
    \le
    C \norm[0]{ w }_{H^1(\Omega)}
    \le
    C \norm[0]{ e_{p}^h }_{\Omega}.
  \end{equation}
  Furthermore, by \cref{eq:Pitilde-approx-ineq2}, we note that
  \begin{equation*}
      \norm[0]{\tilde{\Pi}w}_{K}
      \le \norm[0]{\tilde{\Pi}w - w}_K + \norm[0]{w}_K
      \le Ch_K\norm[0]{w}_{H^1(K)} + \norm[0]{w}_K \le C\norm[0]{w}_{H^1(K)},
  \end{equation*}
  and so
  $\norm[0]{\tilde{\Pi}w}_{\Omega} \le C\norm[0]{w}_{H^1(\Omega)} \le
  C \norm[0]{e_p^h}_{\Omega}$. Using this inequality and
  \cref{eq:pitildewneph} in \cref{eq:checkingstuff} we find that
  $\norm[0]{\tilde{\Pi}w}_{V_h}^2 \le
  C\norm[0]{e_p^h}_{\Omega}^2$. Therefore, in combination with
  \cref{eq:firstboundI2},
  \begin{equation*}
      |I_2| \le 
      C\max\cbr[0]{\kappa_{\min}^{-1},\alpha_f^{-1}} \norm[0]{e_u^h}_{V_h} \norm[0]{e_p^h}_{\Omega},
  \end{equation*}
  and by the Cauchy-Schwarz inequality,
  $|I_3| \le C \kappa_{\min}^{-1} \norm[0]{ e_{u}^I }_{\Omega}
  \norm[0]{ e_{p}^h }_{\Omega}$. Combining \cref{eq:normepihI1I2I3}
  with the bounds for $I_1$, $I_2$, $I_3$, and using Young's
  inequality,
  \begin{equation*}
    \norm[0]{e_{p}^h}_{\Omega}^2
    \le C h \sum_{K \in \mathcal{T}} \norm[0]{e_{p}^h-\bar{e}_{p}^h}_{\partial K \setminus (\Gamma_c\cup \Gamma_s)}^2
    + C (\max \cbr[0]{ \kappa_{\min}^{-1}, \alpha_f^{-1} })^2 \norm[0]{ {e}_u^h }_{{V}_h}^2
    + C \kappa_{\min}^{-2} \norm[0]{ e_{u}^I }_{\Omega}^2.
  \end{equation*}
  Applying \cref{eq:u-pf-h-estimate},
  \begin{equation*}
    \begin{split}
      \norm[0]{ e_{p}^h }_{\Omega}^2
      \le& \sbr[3]{Ch \alpha^{-1} C_c^{-1}
        +  C (\max \cbr[0]{ \kappa_{\min}^{-1}, \alpha_f^{-1} })^2 C_c^{-2}}
      \del[3]{ \kappa_{\min}^{-1} \norm[0]{ e_{u}^I }_{\Omega} + \norm[0]{e_{p_f}^I}_{\Gamma_c} }^2
      \\
      & + C \kappa_{\min}^{-2} \norm[0]{ e_{u}^I }_{\Omega}^2.
    \end{split}
  \end{equation*}
  \Cref{eq:p-estimate} now follows by the triangle inequality,
  \cref{eq:pf-interpolation-L2-error,eq:HDG-projection-u-estimate,eq:HDG-projection-p-estimate}.
\end{proof}

\begin{remark}
  \label{rem:dualityremark}
  The error estimates \cref{eq:u-pf-estimate,eq:p-estimate} show that
  if $s = k+1$ and $s_f = k_f$, then the errors in the velocity and
  pressures in the $L^2$-norm are $\mathcal{O}(h^{\min(k+1,k_f+1)})$
  if the duality argument \cref{eq:pf-interpolation-duality-error}
  holds on the fault, and $\mathcal{O}(h^{\min(k+1,k_f)})$
  otherwise. Hence, if the duality argument holds it suffices to
  choose $k_f = k$ in $Q_h^f$ to obtain errors of
  $\mathcal{O}(h^{k+1})$. If the duality argument does not hold, then
  errors of $\mathcal{O}(h^{k+1})$ can be obtained by choosing
  $k_f = k+1$.
\end{remark}

%------------------------------------------------------------------------------
\section{Numerical results}
\label{s:numericalresults}

We showcase the performance of the proposed HDG/DG scheme in various
examples below\footnote{Further numerical experiments using two
  alternative hybridizable discretizations based on a similar
  philosophy as the HDG/DG scheme, where the inter-dimensional
  coupling between the problems is naturally handled by the face
  variables, are presented in \ref{ap:numres-alt} without analysis.}.
In \cref{ex:mms} we verify the convergence properties of the
discretization established in \cref{thm:errorestimates}. Afterwards,
in \cref{ex:conducting,ex:sealing,ex:intersecting}, we apply the
HDG/DG method to various benchmark problems to demonstrate its
robustness in different scenarios, including whether the fault is
conducting or sealing, whether it is fully, partially, or not
immersed, and whether there are multiple faults that may or may not
intersect.  All the examples were implemented using the
FEniCS~\cite{logg2012automated}-based library
FEniCS\textsubscript{ii}~\cite{kuchta2020assembly} for coupled
inter-dimensional problems.

%------------------------------------------------------------------------------
\subsection{Error convergence}
\label{ex:mms}

We consider a domain $\Omega=(-1, 1)^2$ split into subdomains
$\Omega_1$, $\Omega_2$, $\Omega_3$ by a pair of vertical faults; a
conducting fault $\Gamma_c$ given by the line segment that connects
points $[-0.5, -1]$ and $[-0.5, 1]$ and a sealing fault $\Gamma_s$
defined by points $[0.5, -1]$ and $[0.5, 1]$.  For the exact solution,
we choose $p$ and $p_f$ as follows
\begin{equation*}
  p = \begin{cases}
    \sin(\pi(x+y)) & [x,  y]\in \Omega_1\\
    \cos(\pi(x+y)) & [x,  y]\in \Omega_2\\
    \cos(\pi(2x-y)) & [x,  y]\in \Omega_3\\    
    \end{cases},\quad
  p_f = \sin(\pi(x-2y))
\end{equation*}
and adjust the right-hand sides of
\cref{eq:fault_problem_d,eq:fault_problem_e,eq:fault_problem_f,eq:fault_problem_g}
accordingly. We prescribe $p$ and $p_f$ as boundary data on the top
and bottom edges of $\Omega$ and
$\partial\Gamma_c = (\partial\Gamma_c)^D$, respectively. The
corresponding normal flux $u \cdot n$ is set on the lateral edges. A
sample mesh used in the (uniform) refinement study is shown in
\cref{fig:fault_mms}. In the experiments we let $\alpha_f=2$,
$\xi=0.75$, $\kappa_f=3$ and consider a discontinuous permeability
$\kappa$ such that $\kappa|_{\Omega_1}=5$, $\kappa|_{\Omega_2}=4$ and
$\kappa|_{\Omega_3}=6$.  The stabilization parameter is set as
$\sigma=10 k^2$.

Errors in $u$, $p$, and $p_f$ in the $L^2$-norm are shown for
different levels of refinement and for $k=1,2,3$, $k_f=k$ in
\cref{tab:hdg_ipdg_cvrg} together with rates of convergence. We
observe that the errors for all the unknowns approach the order
$\mathcal{O}(h^{k+1})$, as expected from \cref{thm:errorestimates}.

\begin{figure}
  \centering
  \includegraphics[align=c,height=0.3\textwidth]{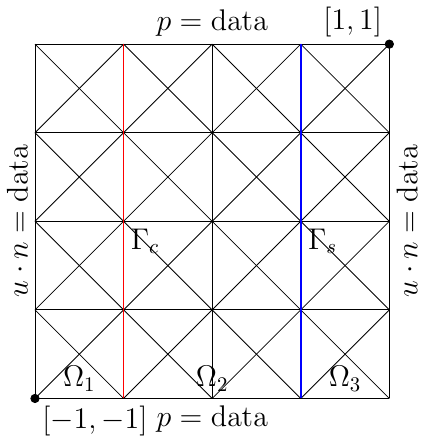}
  \hspace{2pt}
  \includegraphics[align=c, height=0.3\textwidth]{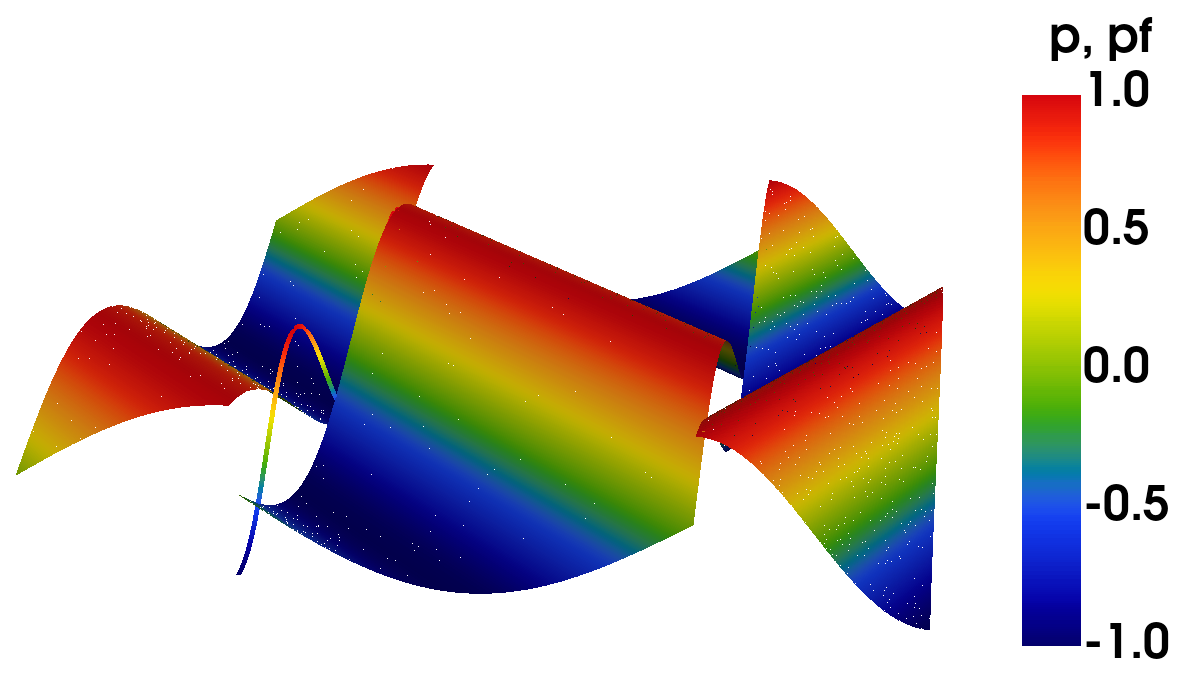}  
  \caption{Setup of the error convergence study in
    \cref{ex:mms}. (Left) Geometry with a conducting fault in red,
    sealing fault in blue, and a sample mesh. (Right) Discrete
    pressures on the finest mesh ($h=1/64$) with the HDG/DG
    discretization choosing $k=k_f=1$.}
  \label{fig:fault_mms}
\end{figure}

\begin{table}
  \centering
  \scriptsize{
    \begin{tabular}{l|lll}
      \hline
      $h$         & $\lVert u - u_h \rVert_{\Omega}$ & $\lVert p - p_h\rVert_{\Omega}$ & $\lVert p_f - p_{h, f}\rVert_{\Gamma_c}$ \\
      \hline
      \multicolumn{4}{c}{$k=1$}\\
      \hline
 0.5    & 6.984E+00(--) & 8.934E-01(--) & 2.026E-01(--) \\
 0.25   & 1.842E+00(1.92)  & 2.010E-01(2.15)  & 1.943E-01(0.06)  \\
 0.125  & 4.654E-01(1.98)  & 4.715E-02(2.09)  & 5.082E-02(1.93)  \\
 0.0625 & 1.165E-01(2.00)  & 1.129E-02(2.06)  & 1.309E-02(1.96)  \\
 0.0312 & 2.909E-02(2.00)  & 2.754E-03(2.04)  & 3.323E-03(1.98)  \\
 0.0156 & 7.267E-03(2.00)  & 6.795E-04(2.02)  & 8.371E-04(1.99)  \\
  \hline
  \multicolumn{4}{c}{$k=2$}\\
  \hline
 0.5    & 1.470E+00(--) & 1.522E-01(--) & 1.992E-01(--) \\
 0.25   & 1.888E-01(2.96)  & 2.001E-02(2.93)  & 2.105E-02(3.24)  \\
 0.125  & 2.366E-02(3.00)  & 2.373E-03(3.08)  & 2.759E-03(2.93)  \\
 0.0625 & 2.950E-03(3.00)  & 2.863E-04(3.05)  & 3.484E-04(2.99)  \\
 0.0312 & 3.680E-04(3.00)  & 3.508E-05(3.03)  & 4.365E-05(3.00)  \\
 0.0156 & 4.594E-05(3.00)  & 4.338E-06(3.02)  & 5.459E-06(3.00)  \\
  \hline
  \multicolumn{4}{c}{$k=3$}\\
  \hline
 0.5    & 2.385E-01(--) & 2.981E-02(--) & 6.027E-03(--) \\
 0.25   & 1.573E-02(3.92)  & 1.624E-03(4.20)  & 1.923E-03(1.65)  \\
 0.125  & 9.905E-04(3.99)  & 9.712E-05(4.06)  & 1.238E-04(3.96)  \\
 0.0625 & 6.190E-05(4.00)  & 5.893E-06(4.04)  & 7.829E-06(3.98)  \\
 0.0312 & 3.864E-06(4.00)  & 3.622E-07(4.02)  & 4.916E-07(3.99)  \\
 0.0156 & 2.413E-07(4.00)  & 2.244E-08(4.01)  & 3.079E-08(4.00)  \\
      \hline
    \end{tabular}
  }
  \caption{Convergence of the HDG/DG approximation with $k_f=k=1,2,3$
    to the solution of the coupled problem setup in
    \cref{ex:mms}. Errors are reported in the $L^2$-norms on the
    respective domains.  Estimated rates are shown in the brackets.  }
  \label{tab:hdg_ipdg_cvrg}
\end{table}

\begin{remark}[Non-convex $\Gamma_c$]\label{rmrk:nonconvex}
  In the previous example the convexity of $\Gamma_c$ ensured that the
  estimate \eqref{eq:pf-interpolation-duality-error} held and in turn,
  following \cref{thm:errorestimates}, taking $k_f=k$ was sufficient
  to achieve convergence order $k+1$. Here we consider a setup where
  $\Gamma_c$ is not convex and the duality argument may not hold.
  
  We let $\Omega=(-1, 1)\times (-0.5, 1.5)$ and split the domain into
  $\Omega_+$, $\Omega_-$ by a single conducting fracture $\Gamma_c$
  defined as a union of three segments $\Gamma_{c, 1}$,
  $\Gamma_{c, 2}$ and $\Gamma_{c, 3}$, cf. \cref{fig:nonconvex}.  We
  define the exact solution $p$ and $p_f$ as follows
  \begin{equation*}
    p|_{\Omega_{+}} = \sin(\pi(x+y)),\quad
    p|_{\Omega_{-}} = \cos(\pi(x/2+y)),\quad
    p_f = \begin{cases}
      \cos(2 \pi y) & [x,  y]\in \Gamma_{c, 2}\\        
      \cos(\pi x) & [x,  y]\in \Gamma_{c, 1}\cup\Gamma_{c, 3}\\    
    \end{cases}.
  \end{equation*}
  We prescribe $p$ and $p_f$ as boundary data on $\partial\Omega$ and
  $\partial\Gamma_c = (\partial\Gamma_c)^D$, respectively and set the
  model parameters as $\alpha_f=2$, $\xi=0.75$, $\kappa_f=3$ while
  considering a discontinuous permeability $\kappa$ such that
  $\kappa|_{\Omega_+}=5$, $\kappa|_{\Omega_-}=4$. Accordingly we also
  adjust the right-hand sides of
  \cref{eq:fault_problem_d,eq:fault_problem_e}.

  Setting $k_f=k$, $k=1, 2, 3$ and (as before) $\sigma=10 k^2$ we
  observe in \cref{tab:nonconvex} that while $\Gamma_c$ is non-convex,
  the HDG/DG scheme still yields convergence rates of order $k+1$ for
  $u$, $p$ and $p_f$ without the need to choose $k_f=k+1$.

  \begin{figure}
    \centering
    \includegraphics[align=c, height=0.3\textwidth]{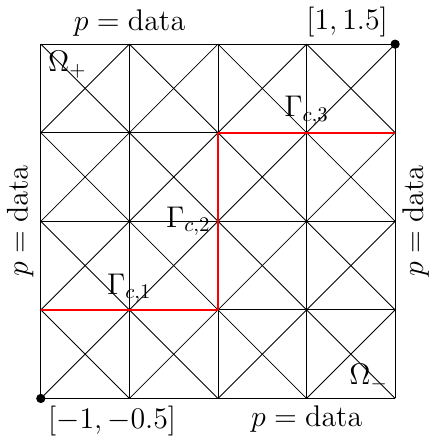}
    \hspace{2pt}
    \includegraphics[align=c, height=0.3\textwidth]{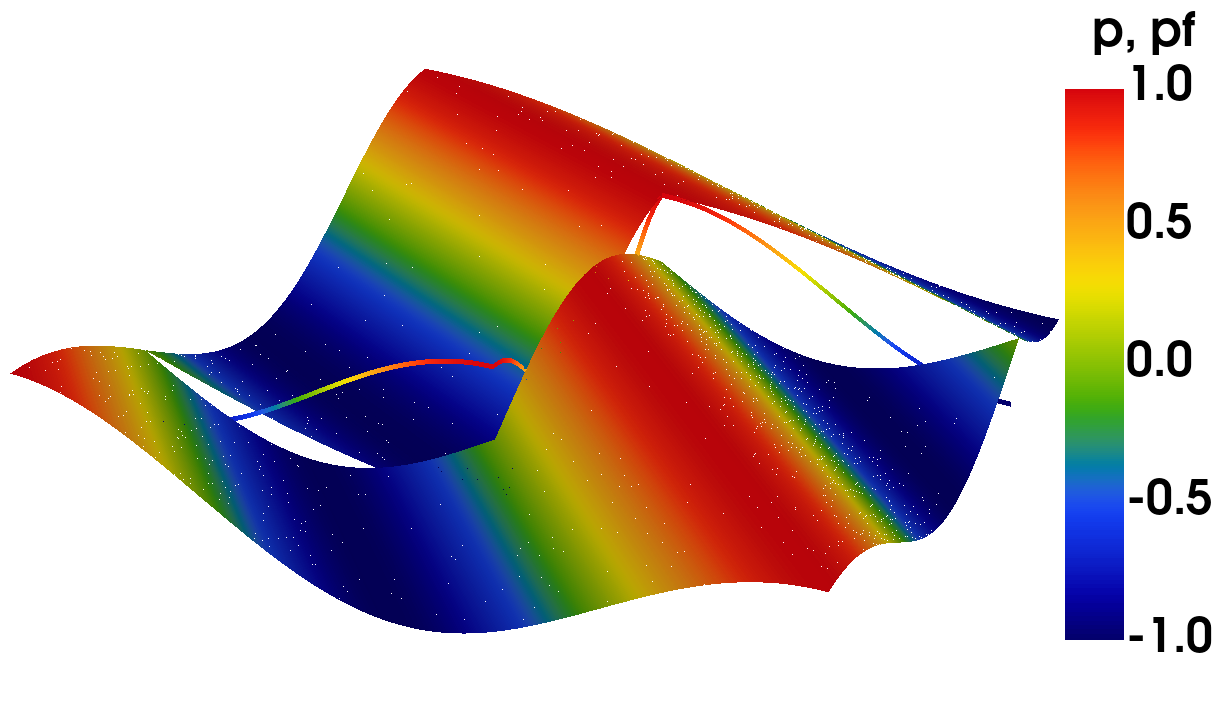}  
    \caption{Setup of the error convergence study in
      \cref{rmrk:nonconvex}. (Left) Geometry with a non-convex
      conducting fault in red, and a sample mesh. (Right) Discrete
      pressures on the finest mesh ($h=1/64$) with the HDG/DG
      discretization choosing $k=k_f=1$.}
    \label{fig:nonconvex}
  \end{figure}
  
  \begin{table}
    \centering
    \scriptsize{
      \begin{tabular}{l|lll}
        \hline
        $h$         & $\lVert u - u_h \rVert_{\Omega}$ & $\lVert p - p_h\rVert_{\Omega}$ & $\lVert p_f - p_{h, f}\rVert_{\Gamma_c}$ \\
        \hline
        \multicolumn{4}{c}{$k=1$}\\
        \hline
        0.5    & 2.288E+00(--) & 2.779E-01(--) & 2.269E-01(--) \\
        0.25   & 6.244E-01(1.87)  & 6.419E-02(2.11)  & 1.544E-01(0.56)  \\
        0.125  & 1.600E-01(1.96)  & 1.510E-02(2.09)  & 4.020E-02(1.94)  \\
        0.0625 & 4.030E-02(1.99)  & 3.639E-03(2.05)  & 9.959E-03(2.01)  \\
        0.0312 & 1.010E-02(2.00)  & 8.916E-04(2.03)  & 2.497E-03(2.00)  \\
        0.0156 & 2.526E-03(2.00)  & 2.206E-04(2.02)  & 6.252E-04(2.00)  \\
        \hline
        \multicolumn{4}{c}{$k=2$}\\
        \hline
        0.5    & 3.617E-01(--) & 3.519E-02(--) & 1.483E-01(--) \\
        0.25   & 3.885E-02(3.22)  & 3.990E-03(3.14)  & 1.538E-02(3.27)  \\
        0.125  & 4.656E-03(3.06)  & 4.685E-04(3.09)  & 1.973E-03(2.96)  \\
        0.0625 & 5.704E-04(3.03)  & 5.647E-05(3.05)  & 2.505E-04(2.98)  \\
        0.0312 & 7.065E-05(3.01)  & 6.921E-06(3.03)  & 3.136E-05(3.00)  \\
        0.0156 & 8.793E-06(3.01)  & 8.564E-07(3.01)  & 3.921E-06(3.00)  \\
        \hline
        \multicolumn{4}{c}{$k=3$}\\
        \hline
        0.5    & 2.748E-02(--) & 3.359E-03(--) & 4.758E-03(--) \\
        0.25   & 2.019E-03(3.77)  & 1.916E-04(4.13)  & 1.384E-03(1.78)  \\
        0.125  & 1.191E-04(4.08)  & 1.124E-05(4.09)  & 8.906E-05(3.96)  \\
        0.0625 & 7.086E-06(4.07)  & 6.773E-07(4.05)  & 5.588E-06(3.99)  \\
        0.0312 & 4.296E-07(4.04)  & 4.151E-08(4.03)  & 3.500E-07(4.00)  \\
        0.0156 & 2.640E-08(4.02)  & 2.568E-09(4.01)  & 2.189E-08(4.00)  \\
        \hline
      \end{tabular}
    }
    \caption{Convergence of the HDG/DG approximation with
      $k_f=k=1,2,3$ to the solution of the coupled problem setup in
      \cref{rmrk:nonconvex} with non-convex conducting
      fracture. Errors are reported in the $L^2$-norms on the
      respective domains. Estimated rates are shown in the brackets.
    }
    \label{tab:nonconvex}
  \end{table}
  
\end{remark}

%------------------------------------------------------------------------------
\subsection{Conducting faults}
\label{ex:conducting}

To illustrate the flexibility of modeling the conducting faults
through
\cref{eq:fault_problem_c,eq:fault_problem_d,eq:fault_problem_e} we
consider two model problems with $\Omega=(0, 1)^2$ having bulk
permeability $\kappa=1$ and which includes a single fracture
characterized by $g_f=1$ and $0.5<\xi\leq 1$.

In the first problem the domain contains an embedded fault extending
from $[0.25, 0.75]$ to $[0.75, 0.25]$ with thickness $d=10^{-3}$ and
$\overline{\kappa}_{f, \tau}=\overline{\kappa}_{f, n}=10^{8}$ (see
\cite{liu2024interior}). Boundary conditions are depicted in
\cref{fig:conducting_embedded}.  The authors in \cite{liu2024interior}
do not apply
\cref{eq:fault_problem_c,eq:fault_problem_d,eq:fault_problem_e} on the
fracture. Instead, they model the conducting fault by
\cref{eq:fault_problem_c} and the pressure continuity condition
$p_{+}=p_{-}=p_f$ on $\Gamma_c$, see also \cite{alboin1999domain}. As
noted in \cite{Martin:2005},
\cref{eq:fault_problem_c,eq:fault_problem_d,eq:fault_problem_e} allow
for pressure discontinuity on the fracture and the continuity
conditions $p_{+}=p_{-}=p_f$ correspond to the limit case
$\alpha_f = +\infty$ in \cref{eq:fault_problem_d,eq:fault_problem_e}
where fracture permeability is assumed to be large and the thickness
of the fracture is assumed to be small. Therefore, since
$\alpha_f=2\times 10^{11}$, for the reference solutions, we use
continuous linear Lagrange elements ($\textit{P}_1$) that enforce the
above mentioned pressure continuity conditions by construction of the
finite element space. In \cref{fig:conducting_embedded} we compare the
HDG/DG solutions to this reference $P_1$ solution by sampling the
pressure along the predefined line $y=0.5$ that crosses the fault. We
observe that the results of the HDG/DG method match well with the
results of the conforming scheme.

In the second problem, inspired by \cite{Martin:2005}, the domain is
split by a vertical fracture with $d=10^{-2}$ and a discontinuous
fracture permeability with
$\overline{\kappa}_{f, \tau}=\overline{\kappa}_{f, n}=1$ on segments
$AB$ and $CD$ and
$\overline{\kappa}_{f, \tau}=\overline{\kappa}_{f, n}=2\cdot 10^{-3}$
in the middle part $BC$. \Cref{fig:bench_mjr} shows this domain
together with boundary conditions. We use the mixed finite element
discretization proposed in \cite{Martin:2005} that uses the lowest
order Raviart--Thomas elements paired with piecewise constant
pressures to obtain a reference solution. In \cref{fig:bench_mjr} we
compare the HDG/DG solutions to the reference solution. As before, we
sample the pressure along a predefined line that crosses the
fault. The results are in very good agreement with
\cite{Martin:2005}. This shows that our method can handle different
values of $\alpha_f$, including the limit case. We note that in both
test problems the results of the HDG/DG scheme show little sensitivity
to the choice of parameter $\xi$. In the following test cases of
conducting faults we therefore set $\xi=0.75$.

\begin{figure}
  \centering
  \includegraphics[width=0.30\textwidth, align=c]{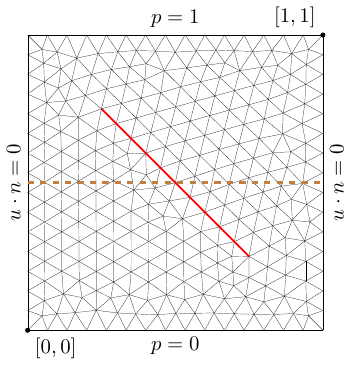}
  \includegraphics[width=0.34\textwidth, align=c]{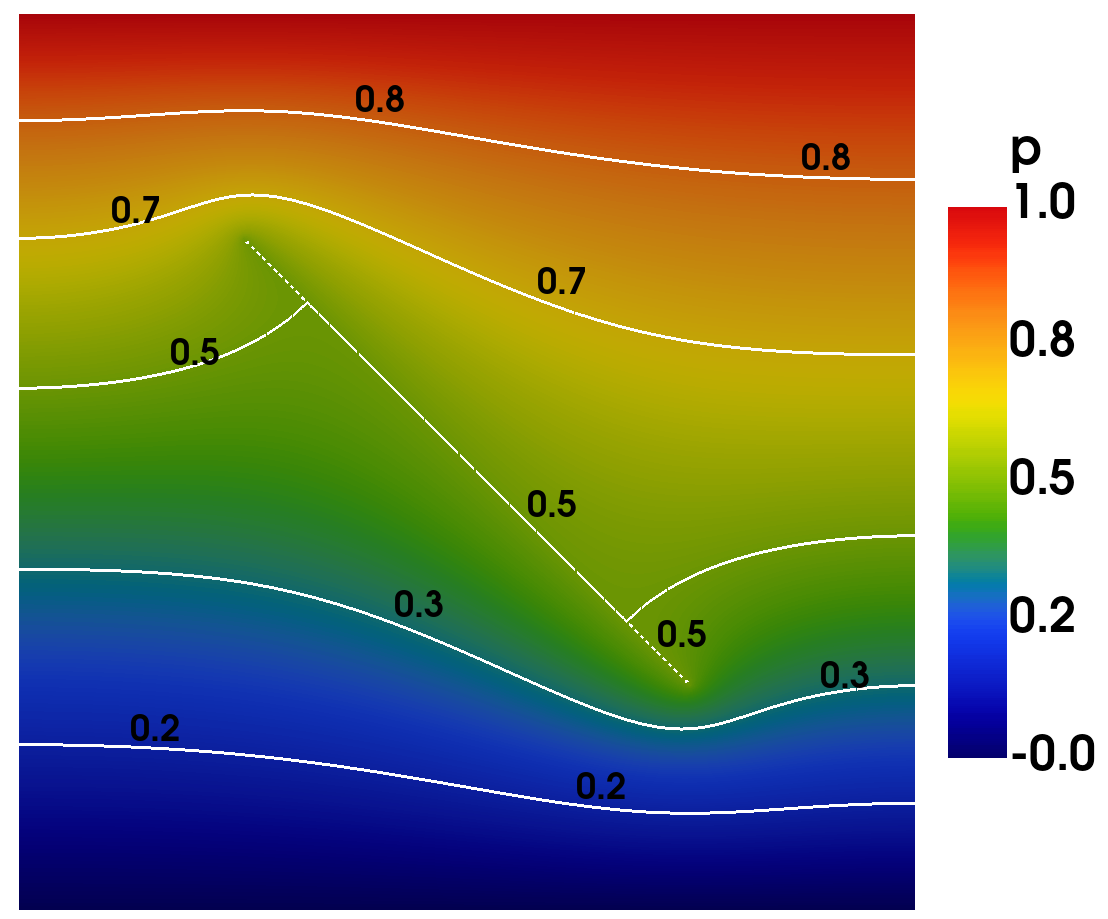}
  \includegraphics[width=0.34\textwidth, align=c]{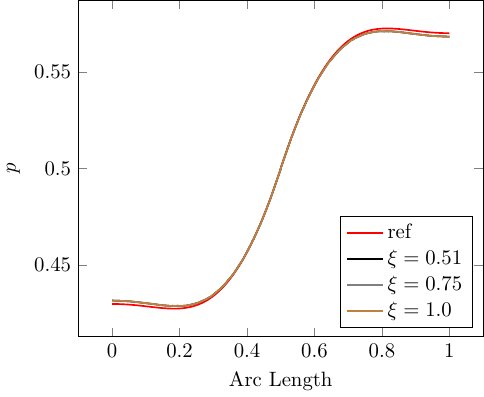}  
  \caption{Fully immersed conducting fault problem described in
    \cref{ex:conducting}. (Left) Problem setup together with the
    initial mesh. The conducting fault is depicted by a red line. To
    compare solutions of different discretizations we sample the
    pressure along the brown dashed line. (Center) The pressure
    distribution obtained by the HDG/DG method with
    $\xi=0.75$. (Right) Comparison of HDG/DG solutions for three
    different values of $\xi$ to a reference $\textit{P}_1$ solution
    of the fracture model \cite{alboin1999domain} in terms of pressure
    values sampled along the brown dashed line in the left panel. The
    conforming scheme reflects the modeling assumption
    $p_{+}=p_{-}=p_f$ on $\Gamma_c$ \cite{alboin1999domain}. Note that
    no value of $\xi$ is needed for the reference solution. The
    HDG/DG solution was computed on a coarse mesh with 2176 cells
    while a fine mesh consisting of 30056 cells is used to compute the
    $\textit{P}_1$ reference solution.  }
  \label{fig:conducting_embedded}
\end{figure}

\begin{figure}
  \centering
  \includegraphics[width=0.30\textwidth, align=c]{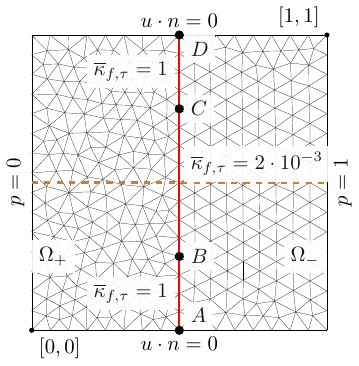}
  \includegraphics[width=0.34\textwidth, align=c]{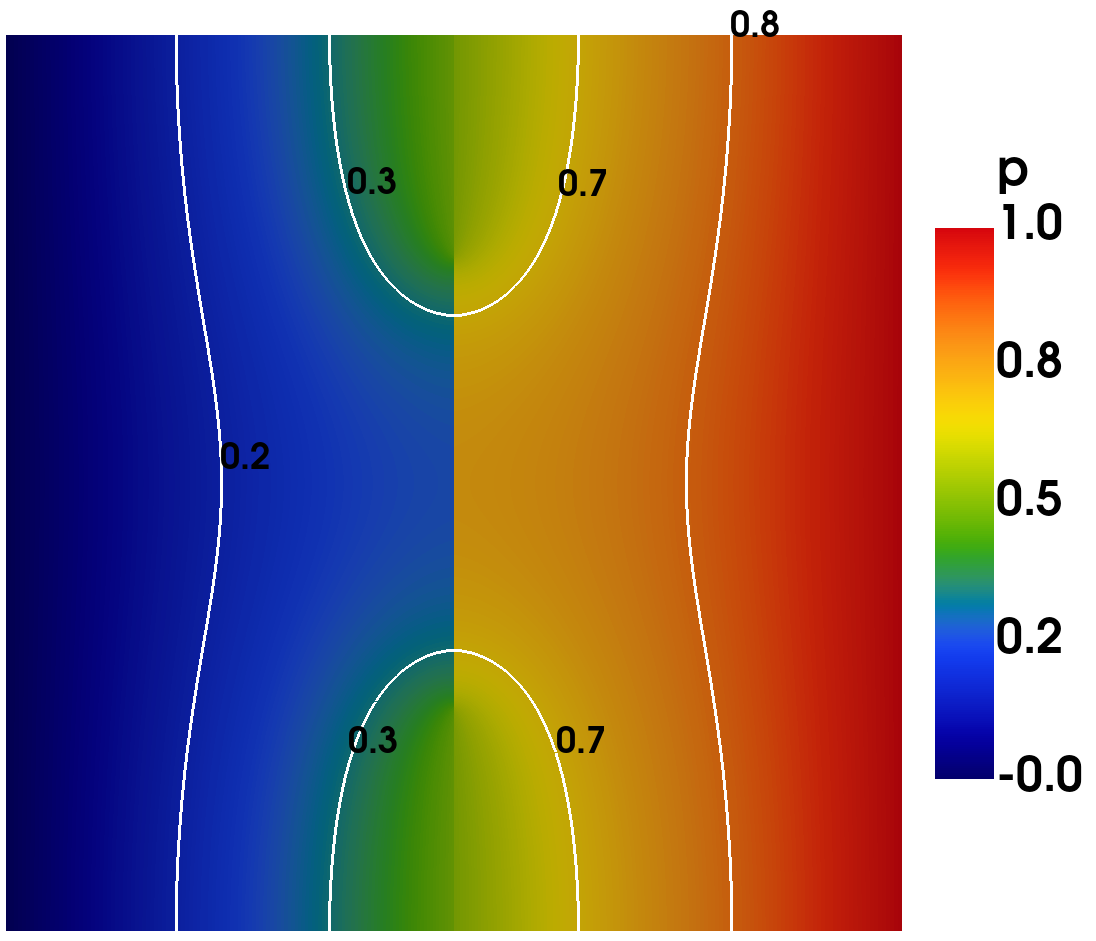}
  \includegraphics[width=0.34\textwidth, align=c]{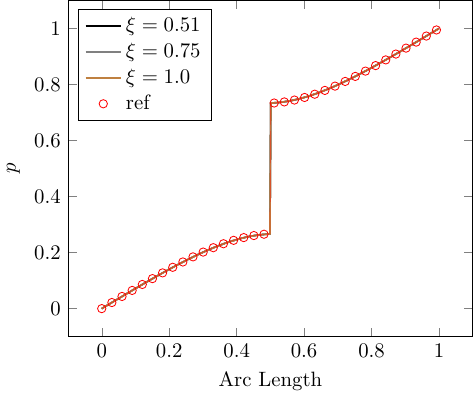}  
  \caption{Conducting fault with discontinuous permeability
    $\overline{\kappa}_{f, n}=\overline{\kappa}_{f, \tau}$ as
    described in \cref{ex:conducting}. (Left) Problem setup together
    with the initial mesh. The conducting fault is depicted by a red
    line. To compare solutions of different discretizations we sample
    the pressure along the brown dashed line. (Center) The pressure
    distribution obtained by the HDG/DG method with $\xi=0.75$.
    (Right) Comparison of HDG/DG solutions for three different values
    of $\xi$ to a reference $\textit{RT}_0$-$\textit{P}_0$ solution
    \cite{Martin:2005} where $\xi=1$ in terms of pressure values
    sampled along the brown dashed line in the left panel. The HDG/DG
    solution was computed on a coarse mesh with 2286 cells while a
    fine mesh with 32354 cells was used for the reference
    solution. For clarity, the reference solution is shown as a thin
    red line as well as using markers.}
  \label{fig:bench_mjr}
\end{figure}

%------------------------------------------------------------------------------
\subsection{Sealing faults}
\label{ex:sealing}

We next consider sealing faults. We consider the domain
$\Omega=(0, 1)^2$ which includes either a single partially immersed
sealing fault extending from $[0.5, 1]$ to $[0.5, 0.5]$ or a fully
immersed sealing fault from $[0.25, 0.75]$ to $[0.75, 0.25]$. We set
the bulk permeability as $\kappa=1$, the fault thickness as
$d=10^{-3}$, and $\overline{\kappa}_{f,n}=10^{-8}$. We prescribe a
pressure gradient in the horizontal direction by setting $p=0$ on the
left boundary and $p=1$ on the right boundary. On the bottom and top
boundaries we impose $u \cdot n = 0$. See the left and center panels
in \cref{fig:blocking_slit,fig:blocking_embedded} for the problem
setup and pressure solution, respectively.

As in \cref{ex:conducting} will compare the pressure solution obtained
with the HDG/DG method ($k=k_f=1$) to a reference solution obtained
with a conforming scheme and on a finer mesh. For the conforming
scheme, however, we now consider a mixed discretization
\cite{lee2022forward} using the $\textit{RT}_0$-$\textit{P}_0$
elements.  This discretization enforces flux conservation on
$\Gamma_s$ by construction. As before the comparison is made by
sampling the pressure along a predefined line that crosses the
fault. From the right most panels in
\cref{fig:blocking_slit,fig:blocking_embedded} we once again observe a
good match of the pressure values and that the results are in good
agreement with the results from \cite{liu2024interior}.

\begin{figure}
  \centering
  \includegraphics[align=c, width=0.30\textwidth]{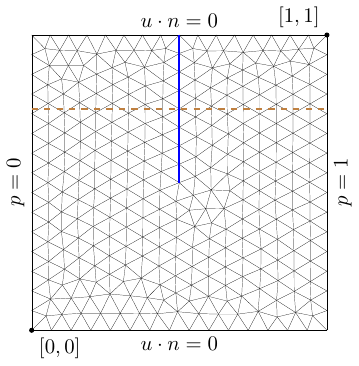}
  \includegraphics[align=c, width=0.34\textwidth, trim={300 0 160 0},clip]{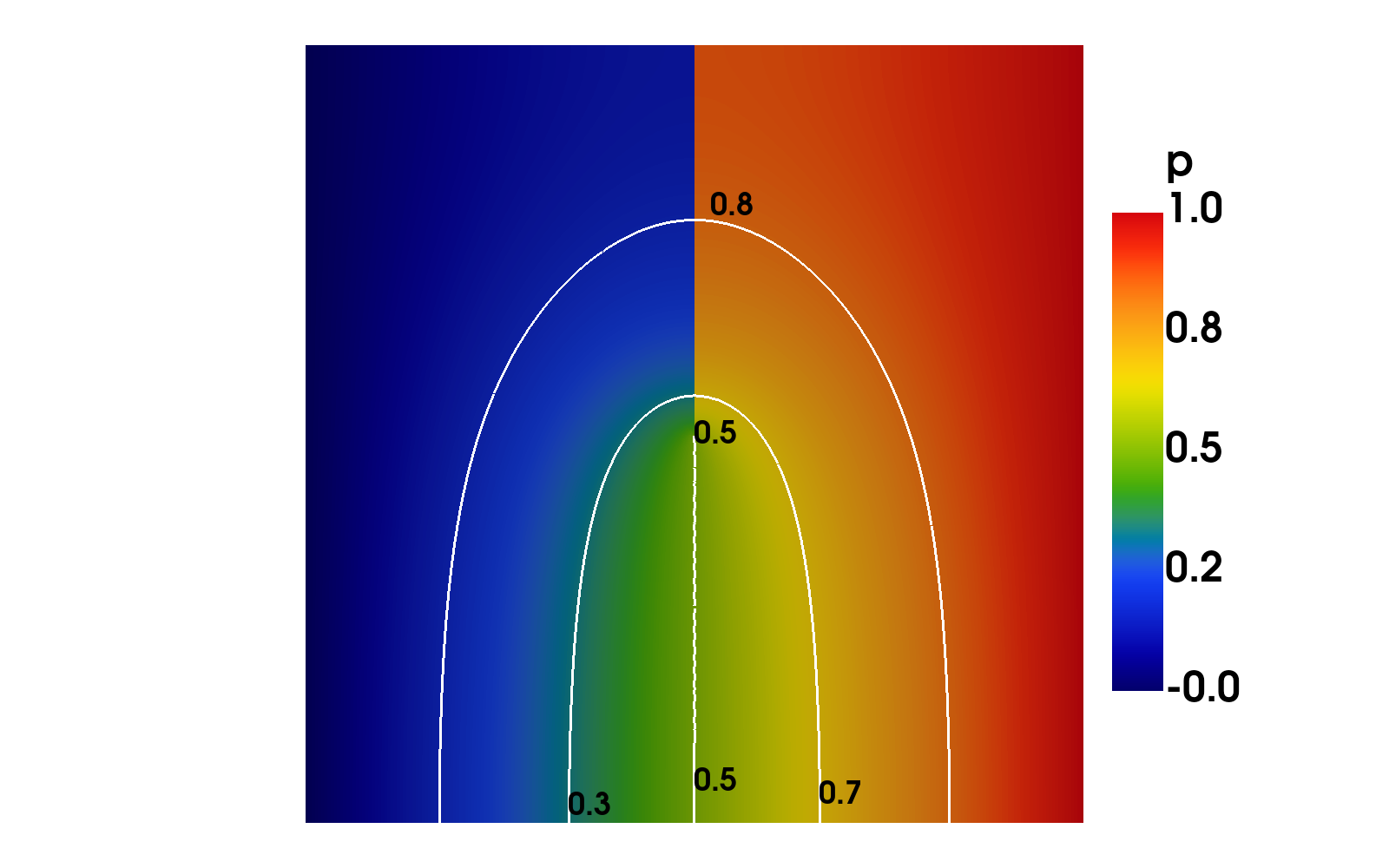}
  \includegraphics[align=c, width=0.34\textwidth]{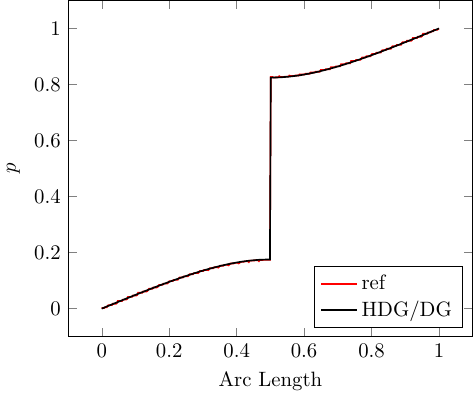}  
  \caption{Partially immersed sealing fault problem as described in
    \cref{ex:sealing}. (Left) Problem setup together with the initial
    mesh. The sealing fault is depicted by a blue line. To compare
    solutions of different discretizations we sample the pressure
    along the brown dashed line. (Center) The pressure distribution
    obtained by the HDG/DG method. (Right) Comparison of HDG/DG to the
    $\textit{RT}_0$-$\textit{P}_0$ reference solution in terms of
    pressure values sampled along the brown dashed line in the left
    panel. The HDG/DG solution was computed on a coarse mesh
    consisting of 2037 cells while a fine mesh consisting of 7703
    cells is used to compute the $\textit{RT}_0$-$\textit{P}_0$
    reference solution.}
  \label{fig:blocking_slit}
\end{figure}

\begin{figure}
  \centering
  \includegraphics[align=c, width=0.30\textwidth]{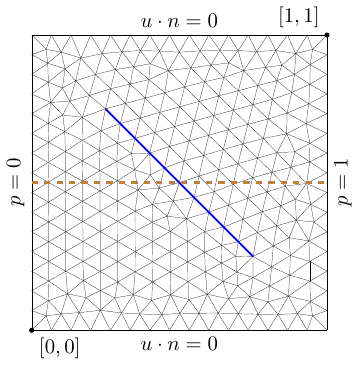}
  \includegraphics[align=c, width=0.34\textwidth, trim={300 0 160 0},clip]{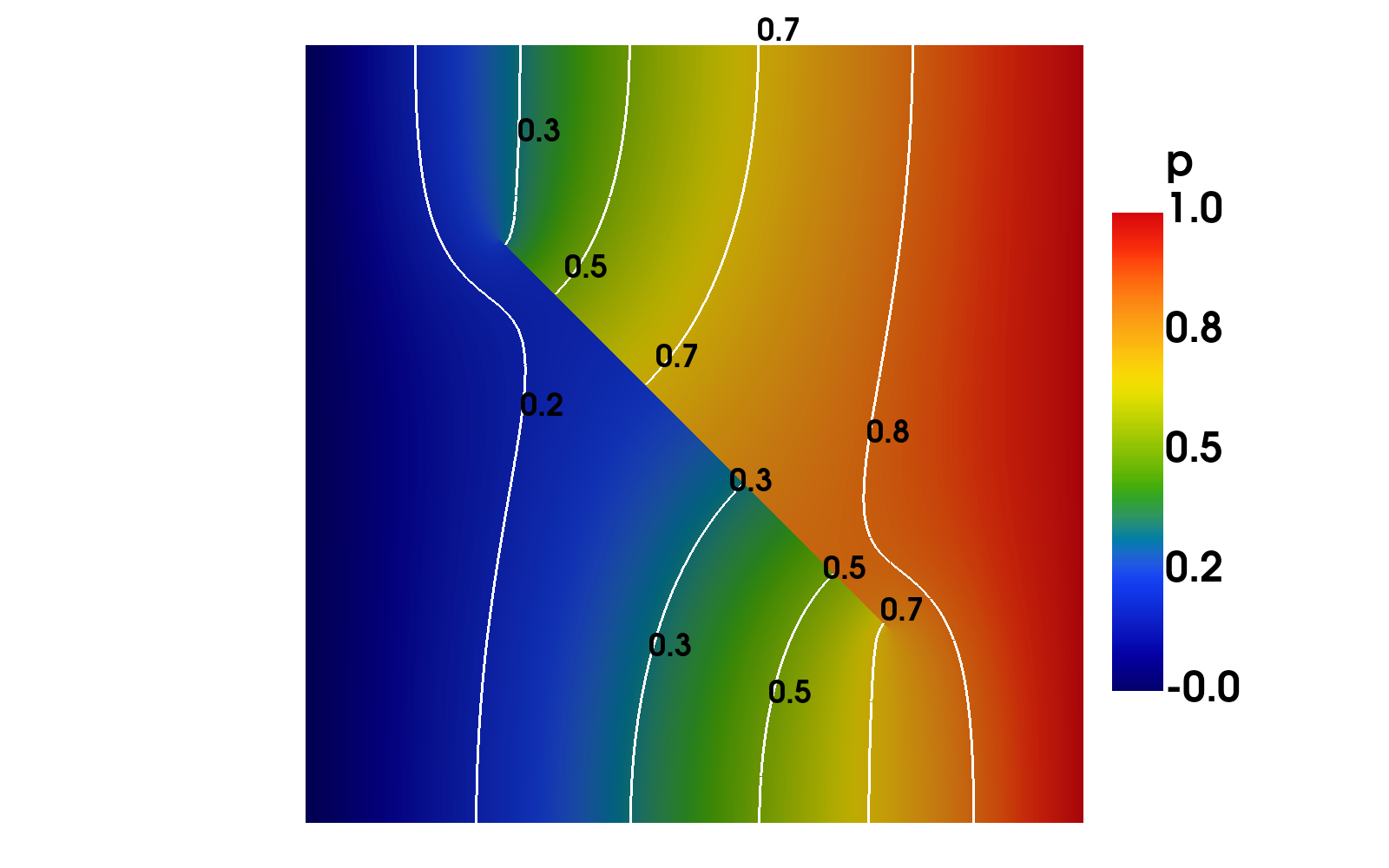}
  \includegraphics[align=c, width=0.34\textwidth]{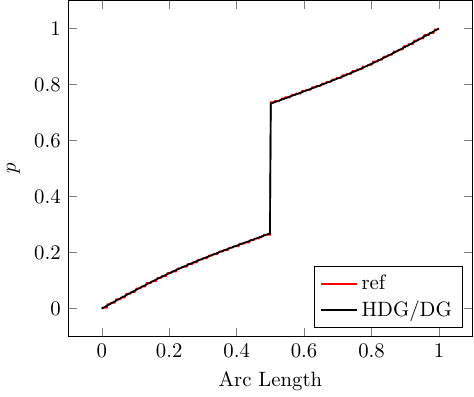}  
  \caption{Fully immersed sealing fault problem as described in
    \cref{ex:sealing}. (Left) Problem setup together with the initial
    mesh. The sealing fault is depicted by a blue line. To compare
    solutions of different discretizations we sample the pressure
    along the brown dashed line. (Center) The pressure distribution
    obtained by the HDG/DG method. (Right) Comparison of HDG/DG to the
    $\textit{RT}_0$-$\textit{P}_0$ reference solution in terms of
    pressure values sampled along the brown dashed line in the left
    panel. The HDG/DG solution was computed on a coarse mesh
    consisting of 1950 cells while a fine mesh consisting of 7580
    cells is used to compute the $\textit{RT}_0$-$\textit{P}_0$
    reference solution.}
  \label{fig:blocking_embedded}
\end{figure}

%------------------------------------------------------------------------------
\subsection{Intersecting faults}
\label{ex:intersecting}

In this final numerical examples section we consider three benchmark
problems from \cite{flemisch2018benchmarks}:
\begin{enumerate}
\item The hydrocoin problem. For this problem the domain is depicted
  in \cref{fig:hydrocoin}. The domain contains two intersecting
  conducting faults with permeability
  $\overline{\kappa}_{f,\tau}=\overline{\kappa}_{f,n}=10^{-6}$. The
  width of the $AB$ fault is $d=5\sqrt{2}$ and the width of the second
  fault is $d=33/\sqrt{5}$. The surrounding medium has permeability
  $\kappa=10^{-8}$.
\item The regular fault network problem. For this problem,
  $\kappa = 1$ and the domain includes axis-aligned conducting faults
  with width $d=10^{-4}$ and
  $\overline{\kappa}_{f,\tau}=\overline{\kappa}_{f,n}=10^{4}$. See
  \cref{fig:geiger}.
\item The complex fault network problem. This problem includes eight
  conducting faults with
  $\overline{\kappa}_{f,\tau}=\overline{\kappa}_{f,n}=10^{4}$ and two
  sealing faults with $\overline{\kappa}_{f,n}=10^{-4}$, see
  \cref{fig:complex}. The width of all faults is $d=10^{-4}$ and we
  set $\kappa=1$ in the bulk.
\end{enumerate}
In all the examples when modeling the conducting faults we set
$\xi=0.75$ in \cref{eq:fault_problem_d,eq:fault_problem_e},
cf. \cref{ex:conducting}, and set $g_f=0$. The boundary conditions for
each of the benchmark problems are specified in the left most panels
of \cref{fig:hydrocoin,fig:geiger,fig:complex}.

As in \cref{ex:conducting,ex:sealing} the comparison of solutions of
the HDG/DG discretization against the reference solutions is done in
terms of pressure solutions along predefined lines crossing (some of)
the faults. For the HDG/DG methods we choose $k=k_f=1$ and relatively
coarse meshes (7705 cells for the hydrocoin problem, 2382 cells for
the regular fault network problem, and 2744 cells for the complex
fault network problem). We use a mimetic finite difference (MFD)
method \cite{brezzi2005family}, on much finer meshes, to compute the
reference solution because this is the method used to obtain the
results in \cite{flemisch2018benchmarks}.

In all three problems the HDG/DG pressure solution agrees well with
the reference pressure solution obtained with the MFD method. The
solution profiles also match well with those given in
\cite{flemisch2018benchmarks}.

\begin{figure}
  \centering
  \includegraphics[align=c, height=0.23\textwidth]{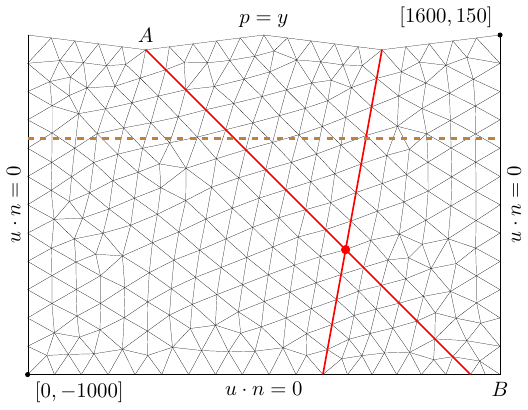}
  \includegraphics[align=c, height=0.23\textwidth]{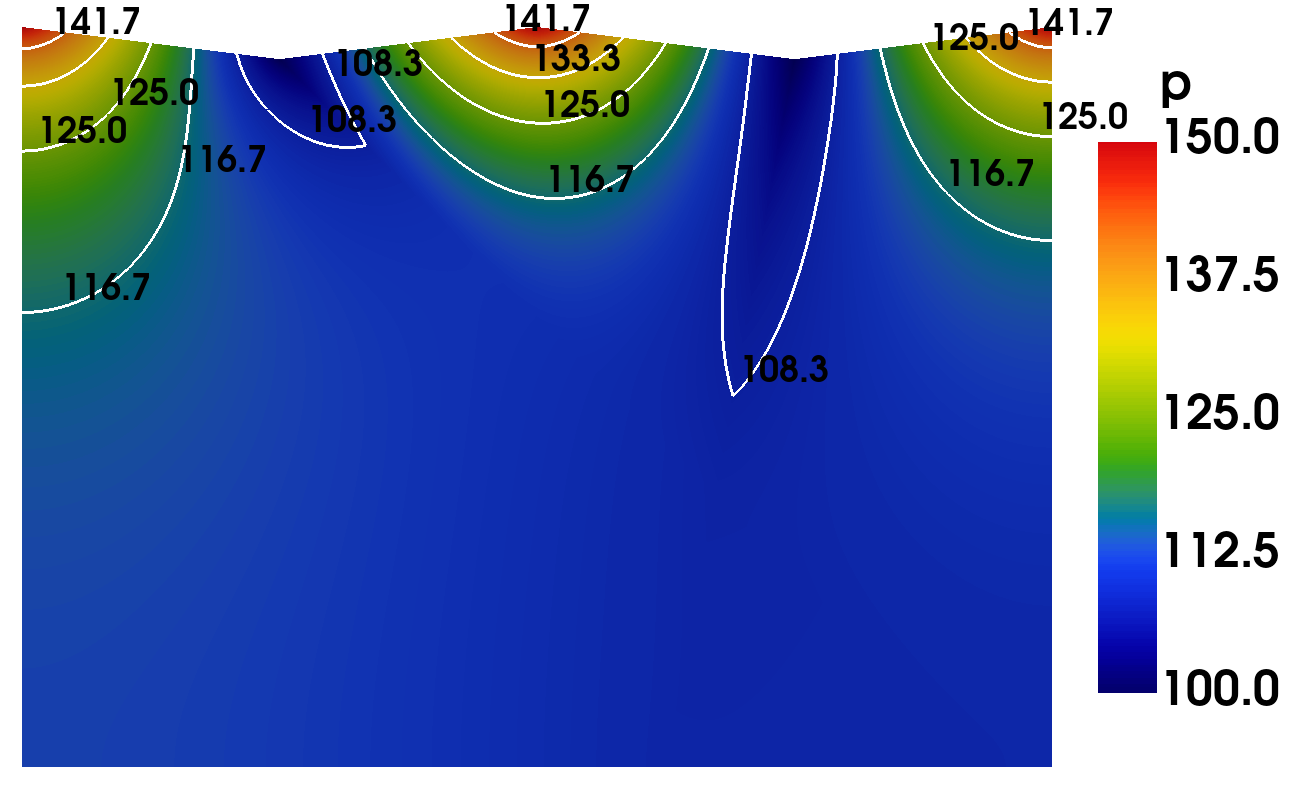}
  \includegraphics[align=c, height=0.24\textwidth]{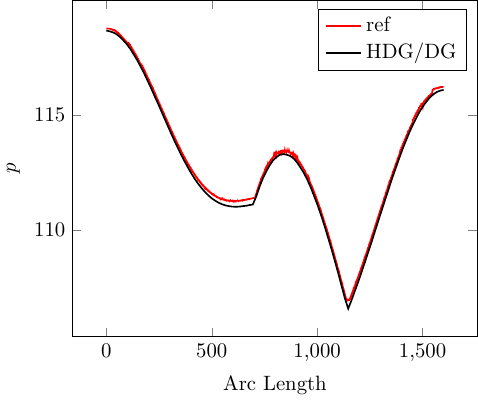}  
  \caption{The hydrocoin problem as described in
    \cref{ex:intersecting}. (Left) Problem setup together with the
    initial mesh. Conducting faults are depicted by red lines. To
    compare solutions of different discretizations we sample the
    pressure along the brown dashed line. (Center) The pressure
    distribution obtained by the HDG/DG method. (Right) Comparison of
    the HDG/DG pressure solution computed on a mesh with 7705 cells to
    the reference pressure solution computed using MFD on a mesh with
    approximately 42500 (mixed - quadrilateral/triangular) cells
    \cite{flemisch2018benchmarks}.}
  \label{fig:hydrocoin}
\end{figure}

\begin{figure}
  \centering
  \includegraphics[align=c, width=0.3\textwidth]{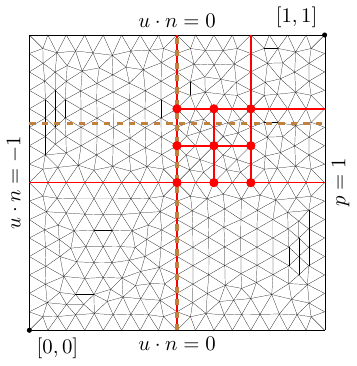}
  \includegraphics[align=c, width=0.32\textwidth]{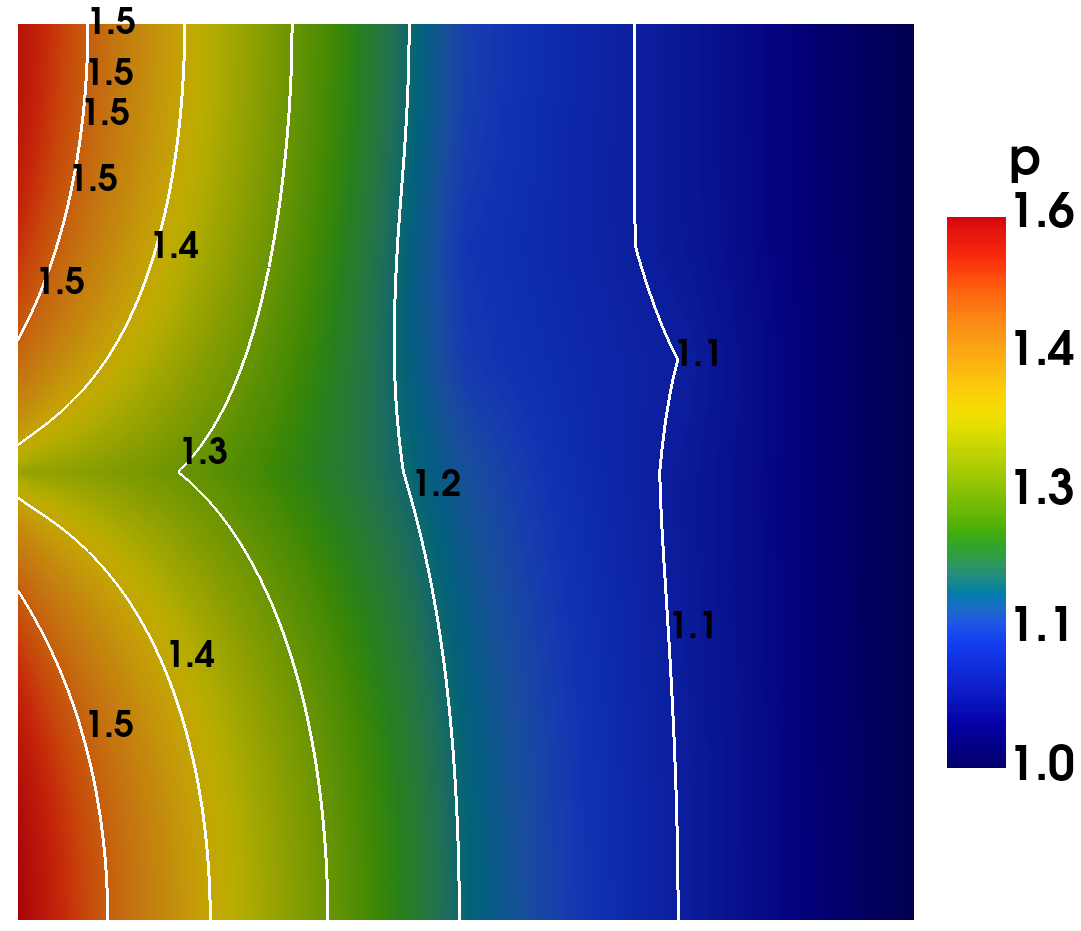}\\
  \includegraphics[align=c, width=0.35\textwidth]{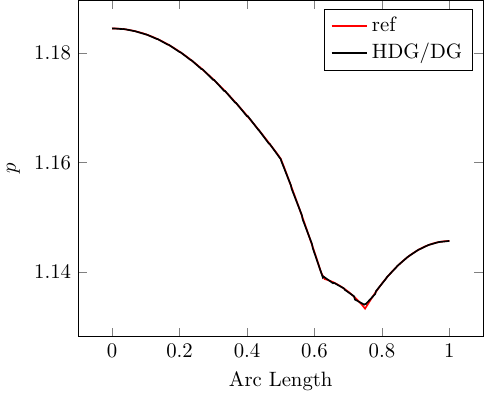}
  \includegraphics[align=c, width=0.35\textwidth]{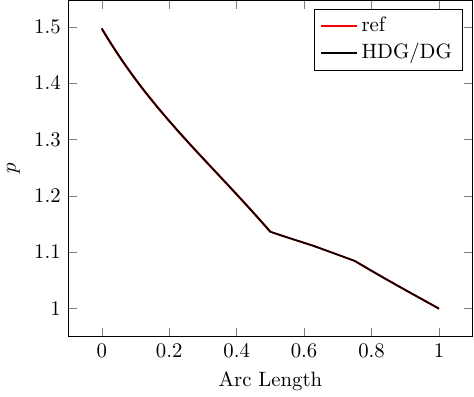}    
  \caption{The regular fault network problem as described in \cref{ex:intersecting}. (Top left) Problem setup
    together with the initial mesh. Conducting faults are depicted by
    red lines. To compare solutions of different discretizations we
    sample the pressure along the brown dashed lines, $x=0.5$, which
    coincides with a fault, and $y=0.75$. (Top right) The pressure
    distribution obtained by the HDG/DG method. (Bottom) Comparison of
    the HDG/DG pressure solution in terms of pressure values on a line
    $x=0.5$ (left) and a line $y=0.75$ (right) computed on a mesh with
    2382 cells to the reference pressure solution computed using MFD
    on a mesh with approximately 1.1 million (triangular) cells
    \cite{flemisch2018benchmarks}.
  }
  \label{fig:geiger}
\end{figure}

\begin{figure}
  \centering
  \includegraphics[align=c, width=0.30\textwidth]{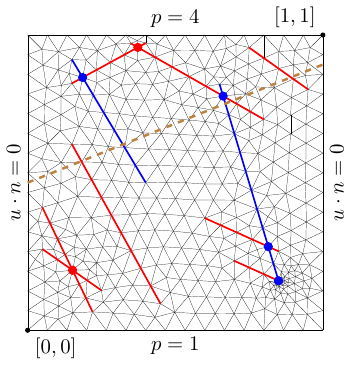}
  \includegraphics[align=c, width=0.34\textwidth]{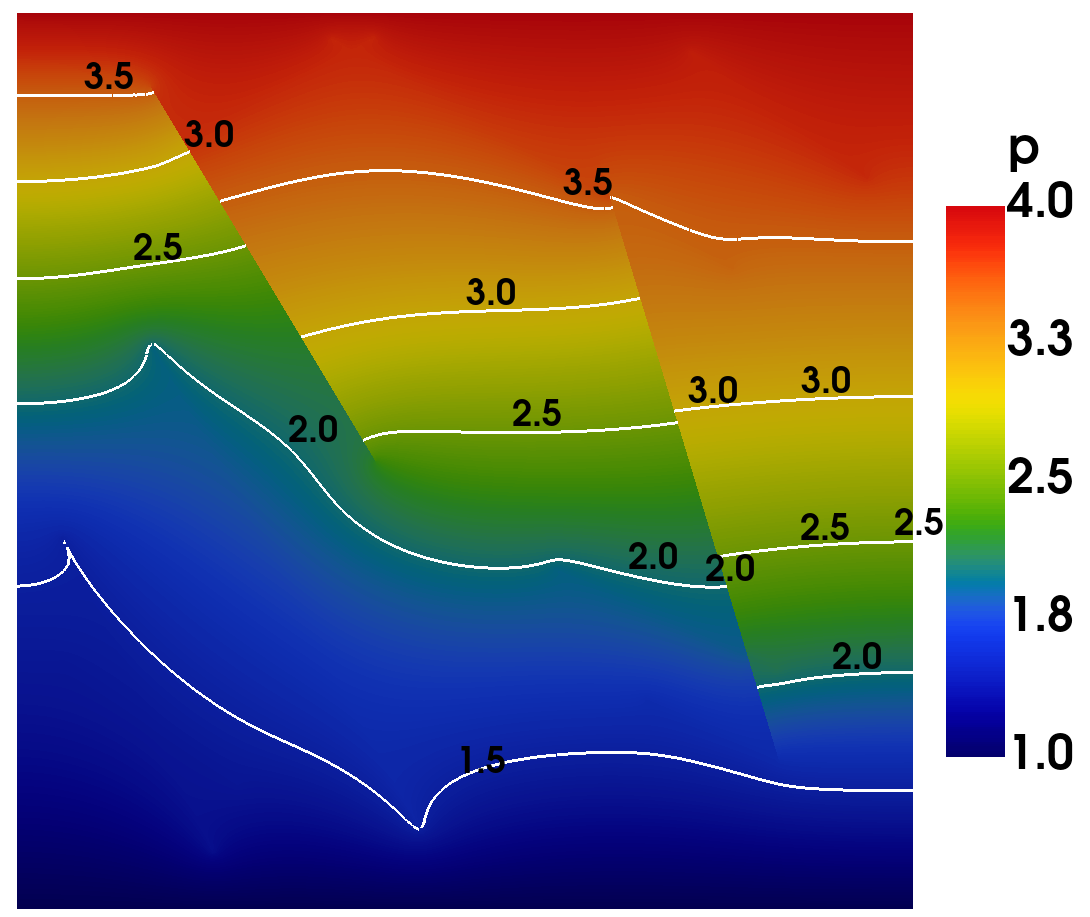}
  \includegraphics[align=c, width=0.34\textwidth]{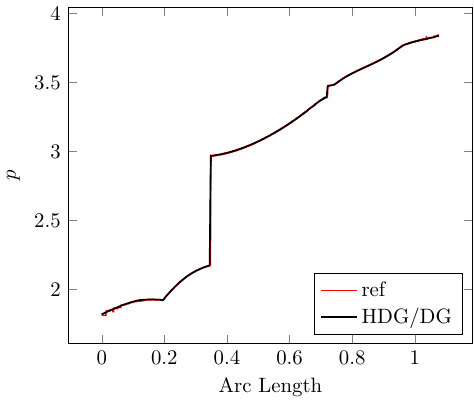}  
  \caption{The complex fault network problem as described in
    \cref{ex:intersecting}. (Left) Problem setup together with the
    initial mesh. Conducting/sealing faults are depicted by red/blue
    lines. To compare solutions of different discretizations we sample
    the pressure along the brown dashed line which crosses multiple
    conducting and sealing faults. (Center) The pressure distribution
    obtained by HDG/DG method. (Right) Comparison of the HDG/DG
    pressure solution computed on a mesh with 2744 cells to the
    reference pressure solution computed using MFD on a mesh with
    approximately 1.2 million (quadrilateral/triangular) cells
    \cite{flemisch2018benchmarks}.}
  \label{fig:complex}
\end{figure}

%------------------------------------------------------------------------------
\section{Conclusions}
\label{s:conclusions}

Due to the introduction of $(\dim-1)$-dimensional face unknowns, the
HDG discretization provides a natural framework to couple
$\dim$-dimensional problems to $(\dim-1)$-dimensional problems. We
exploit this to introduce a discretization for porous media with
faults in which we discretize the $\dim$-dimensional Darcy equations
by an HDG method and the $(\dim-1)$-dimensional fault problem by an
IPDG method. We proved the well-posedness of this discretization and
presented an a priori error analysis. Manufactured solution numerical
examples verify our analysis while benchmark problems show good
comparison of the HDG/DG solution with solutions obtained with other
discretizations. Future work includes the design of iterative solvers
specifically for the hybridized form of the proposed HDG/DG
discretization.

%------------------------------------------------------------------------------
\section*{Acknowledgments}

AC and JJL are supported by the National Science Foundation under
Grant No. DMS-2110782 and DMS-2110781, respectively. MK acknowledges
funding from the Research Council of Norway (grant no. 303362) and SR
was funded by a Discovery Grant from the Natural Sciences and
Engineering Research Council of Canada (RGPIN-2023-03237).

%------------------------------------------------------------------------------
\appendix
%------------------------------------------------------------------------------
\section{Two alternative hybridizable discretizations}
\label{ap:numres-alt}

In \cref{s:hdgipdg} we presented the HDG/DG discretization for the
porous media with faults problem \cref{eq:fault_problem}. We now
present two alternative, but very similar, discretizations. These are:
(i) the hybridizable discontinuous Galerkin/continuous Galerkin
(HDG/CG) method; and (ii) the embedded discontinuous
Galerkin/continuous Galerkin method (EDG/CG).

Both discretizations consider a continuous approximation to the
pressure in the fault; the finite element space $Q^f_h$ in
\cref{eq:fem_spaces} is replaced by
\begin{equation*}
  Q_h^f := \cbr[0]{q_h^f \in C^0(\Gamma_c) \,:\, q_h^f \in P_{k_f}(F) \ \forall F \in \mathcal{F}_c^f}, \quad k_f=k.
\end{equation*}
For the HDG/CG scheme, the remaining spaces in \cref{eq:fem_spaces}
are the same. For the EDG/CG scheme, following
\cite{guzey2007embedded}, we replace $\bar{Q}_h$ in
\cref{eq:fem_spaces} by
\begin{equation*}
  \bar{Q}_h := \cbr[0]{\bar{q}_h \in C^0((\cup_{F \in \mathcal{F}} F)\backslash \Gamma_c) \,:\, \bar{q}_h \in P_k(F) \ \forall F \in \mathcal{F}\backslash \mathcal{F}_c^f}.
\end{equation*}
An illustration of the face spaces of the different discretizations is
shown in \cref{fig:face_spaces} for the conducting fault case.

\begin{figure}
  \begin{center}
    \includegraphics[height=0.22\textwidth]{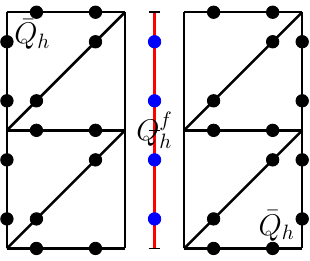}
    \hspace{25pt}
    \includegraphics[height=0.22\textwidth]{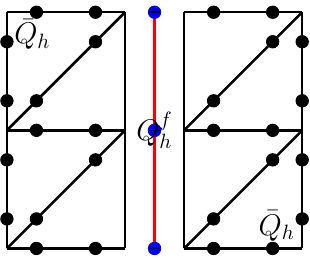}
    \hspace{25pt}    
    \includegraphics[height=0.22\textwidth]{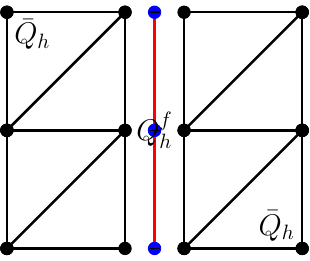}         
    \caption{Schematic of the face function spaces involved in the
      HDG/DG (left), HDG/CG (center), and EDG/CG (right)
      discretizations of a single conducting fault problem. The black
      and blue nodes indicate the degrees of freedom of $\bar{Q}_h$
      and $Q^f_h$, respectively, when linear ($k=k_f=1$) elements are
      used on the faces.}
    \label{fig:face_spaces}
  \end{center}
\end{figure}

\begin{figure}
  \centering
  \includegraphics[align=c,height=0.3\textwidth]{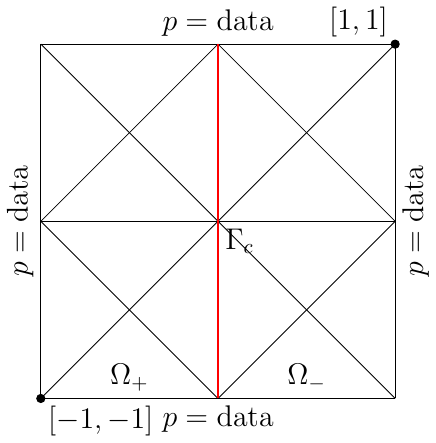}
  \hspace{2pt}
  \includegraphics[align=c, height=0.3\textwidth]{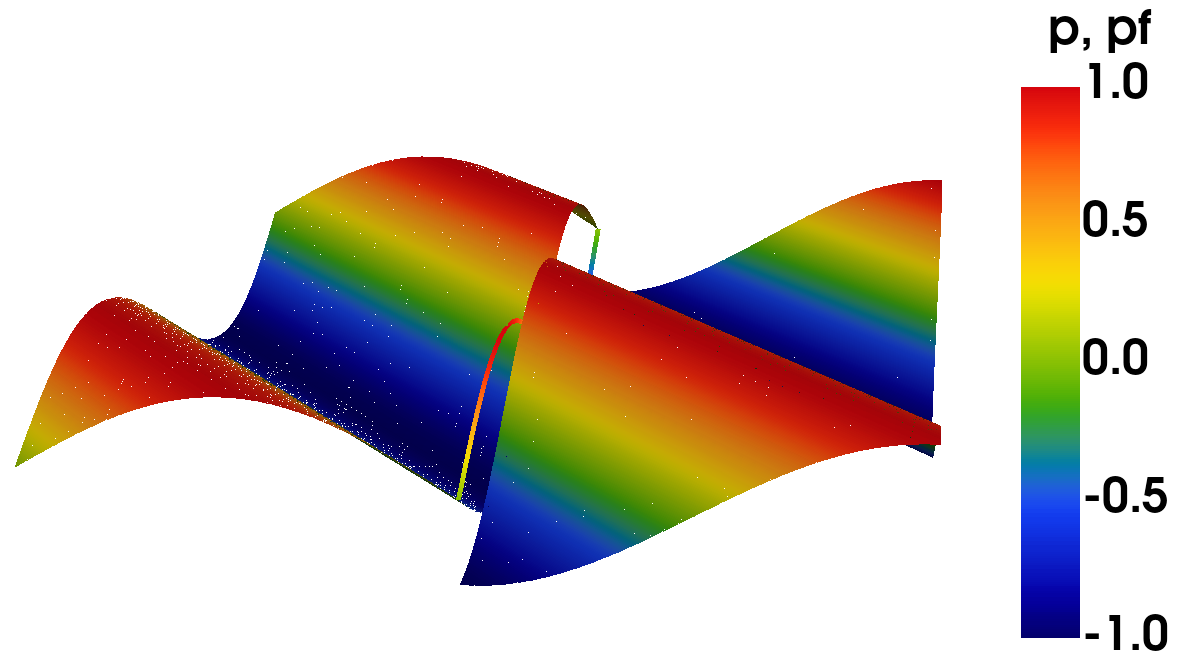}  
  \caption{Setup of the error convergence study for comparing HDG/DG,
    HDG/CG and EDG/CG discretizations. (Left) Geometry with a
    conducting fault in red, and a sample mesh. (Right) Discrete
    pressures on the finest mesh ($h=1/64$) with the HDG/DG
    discretization choosing $k=k_f=1$.}
  \label{fig:method_compare}
\end{figure}

To demonstrate approximation properties of the three discretizations
we consider for simplicity a problem setup with a single conducting
fault. Specifically, we let $\Omega=(-1, 1)^2$ where the domain is
split by a vertical fracture, cf. \cref{fig:method_compare}.  We
define the exact solution $p$ and $p_f$ as follows
\begin{equation*}
  p|_{\Omega_{+}} = \sin(\pi(x+y)),\quad
  p|_{\Omega_{-}} = \cos(\pi(x+y)),\quad
  p_f = \sin(\pi(x-y)).
\end{equation*}
We prescribe $p$ and $p_f$ as boundary data on $\partial\Omega$ and
$\partial\Gamma_c = (\partial\Gamma_c)^D$, respectively. The model
parameters are set as $\alpha_f=2$, $\xi=0.75$, $\kappa_f=3$ and we
take a discontinuous permeability $\kappa$ such that
$\kappa|_{\Omega_+}=5$, $\kappa|_{\Omega_-}=4$.

\Cref{fig:alternative_cvrg} summarizes the errors and rates of
convergence using HDG/DG, HDG/CG and EDG/CG methods with $k=1, 2, 3$
and $k_f=k$. As expected from an EDG discretization, the error in
$u_h$ is $\mathcal{O}(h^k)$ (see, for example,
\cite{Cockburn:2009}). The errors in other variables are
$\mathcal{O}(h^{k+1})$. For HDG/CG, errors in all variables are
$\mathcal{O}(h^{k+1})$, as expected.

\begin{figure}
  \includegraphics[width=0.9\textwidth]{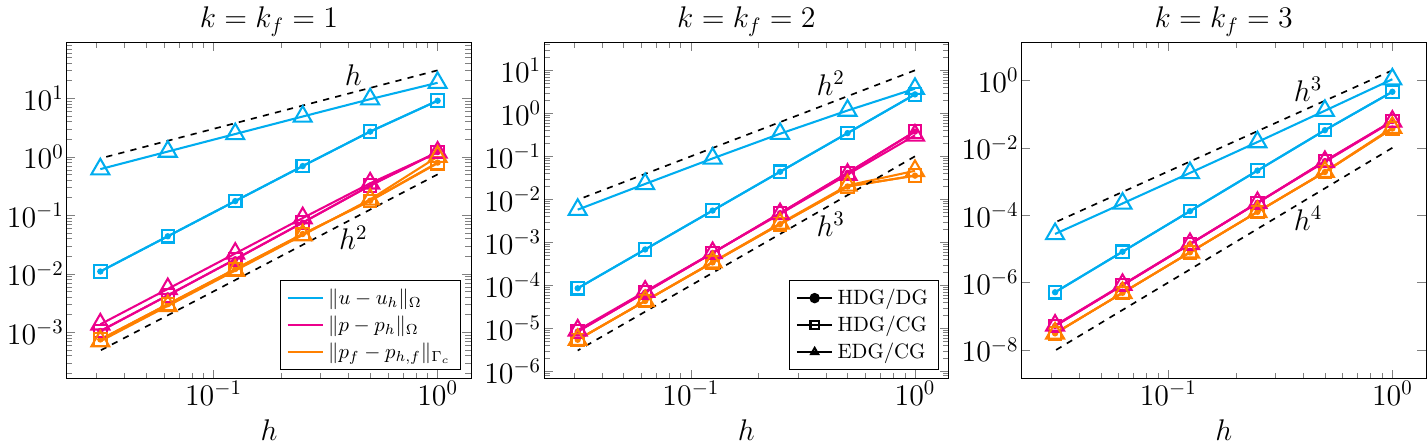}
  \caption{Rates of convergence of the HDG/CG, EDG/CG and HDG/CG
    approximations to the solution of the problem described in
    \ref{ap:numres-alt}. Results due to a particular scheme are
    encoded by a specific marker. Colors represent the errors in the
    different quantities.}
  \label{fig:alternative_cvrg}
\end{figure}

%------------------------------------------------------------------------------
\bibliographystyle{plain}
\bibliography{references}

%------------------------------------------------------------------------------
\end{document}